\RequirePackage{fix-cm}
\documentclass[smallextended]{svjour3}       % onecolumn (second format)
\smartqed  % flush right qed marks, e.g. at end of proof
\usepackage{graphicx}
\usepackage{amsmath}
\usepackage{amsfonts}
\usepackage{cite}
\usepackage{mathtools}
\usepackage[usenames, dvipsnames]{color}

\usepackage{amsthm, upref}

\newcommand{\eq}{\begin{equation}}
\newcommand{\en}{\end{equation}}
\newcommand{\rr}{\mathbb{R}}
\newcommand{\NN}{\mathbb{N}}

\newcommand{\norm}[1]{\left\lVert #1 \right\rVert}

\newcommand{\diri}{\text{Dirichlet}}

\newcommand{\cost}{\mathbf{C}}

\newcommand{\Ent}{\mathrm{Ent}}

\newcommand{\leb}{\mathrm{Leb}}

\newcommand{\simp}{\Delta}
\newcommand{\inc}{\mathcal{I}}
\newcommand{\perm}{\mathcal{S}}

\theoremstyle{plain}
\newtheorem{thm}{Theorem}
\newtheorem{prop}[thm]{Proposition}

\theoremstyle{definition}
\newtheorem{defn}{Definition}
\newtheorem{asmp}{Assumption}
\newtheorem{notn}{Notation}

\theoremstyle{remark}
\newtheorem{rmk}{Remark}
\newtheorem{exm}{Example}

\begin{document}

\title{Multiplicative Schr\"{o}dinger problem and the Dirichlet transport\thanks{S.~Pal's research is supported by NSF grant DMS-1612483. T.-K.~L.~Wong's research is supported by NSERC grant RGPIN-2019-04419.}
}

\titlerunning{Dirichlet transport}        % if too long for running head

\author{Soumik Pal         \and
        Ting-Kam Leonard Wong %etc.
}


\institute{Soumik Pal \at
              Department of Mathematics, University of Washington\\
              \email{soumikpal@gmail.com}           %  \\
           \and
           Ting-Kam Leonard Wong \at
           Department of Statistical Sciences, University of Toronto\\
           \email{tkl.wong@utoronto.ca}
}

\date{Received: date / Accepted: date}

\maketitle

\begin{abstract}
We consider an optimal transport problem on the unit simplex whose solutions are given by gradients of exponentially concave functions and prove two main results. First, we show that the optimal transport is the large deviation limit of a particle system of Dirichlet processes transporting one probability measure on the unit simplex to another by coordinatewise multiplication and normalizing. The structure of our Lagrangian and the appearance of the Dirichlet process relate our problem closely to the entropic measure on the Wasserstein space as defined by von-Renesse and Sturm in the context of Wasserstein diffusion. The limiting procedure is a triangular limit where we allow simultaneously the number of particles to grow to infinity while the `noise' tends to zero. The method, which generalizes easily to many other cost functions, including the squared Euclidean distance, provides a novel combination of the Schr\"odinger problem approach due to C.~L\'eonard and the related Brownian particle systems by Adams et al.~which does not require gamma convergence. Second, we analyze the behavior of entropy along the paths of transport. The reference measure on the simplex is taken to be the Dirichlet measure with all zero parameters which relates to the finite-dimensional distributions of the entropic measure. The interpolating curves are not the usual McCann lines. Nevertheless we show that entropy plus a multiple of the transport cost remains convex, which is reminiscent of the semiconvexity of entropy along lines of McCann interpolations in negative curvature spaces. We also obtain, under suitable conditions, dimension-free bounds of the optimal transport cost in terms of entropy.


\end{abstract}

\section{Introduction} \label{sec:intro} Throughout this paper let $\Delta_n$ be the open unit simplex in $\mathbb{R}^n$ defined by
\begin{equation} \label{eqn:simplex}
\Delta_n = \{p = (p_1, \ldots, p_n) \in (0, 1)^n : p_1 + \cdots + p_n = 1\}, \quad n \geq 2.
\end{equation}
In a series of papers \cite{PW14, W15, PW16, P16, W17} we introduced and studied a Monge-Kantorovich optimal transport problem on the unit simplex with the cost function
\begin{equation} \label{eqn:cost.function}
c(p, q) := \log \left(\frac{1}{n} \sum_{i = 1}^n \frac{q_i}{p_i}\right) - \frac{1}{n}\sum_{i = 1}^n  \log \frac{q_i}{p_i}, \quad p, q \in \Delta_n.
\end{equation}
%While the original motivation comes from mathematical finance, it turns out that this transport problem has remarkable properties that are not adequately covered by the current theory. 
Whereas the quadratic transport on Euclidean space is solved by the gradient map of a convex function (see e.g.~\cite{V03, V08}), our transport problem can be solved in terms of {\it exponentially concave functions}, i.e., functions $\varphi$ such that $e^{\varphi}$ are concave. Exponentially concave functions have been applied to several recent results related to optimal transport. For example, in \cite{EKS15} it was used to prove the equivalence of the entropic curvature-dimension condition and Bochner's inequality.

Given two Borel probability measures $P$ and $Q$ on $\Delta_n$, there exists, under suitable conditions, an exponentially concave function $\varphi$ on $\Delta_n$ whose gradient generates the Monge solution transporting $P$ to $Q$. The details are given in Section \ref{sec:transport.problem}.
%\begin{equation} \label{our.Monge.solution}
%q = T(p), \quad q_i = \frac{1 + \nabla_{e_i - r} \varphi(r)}{\sum_{j = 1}^n (1 + \nabla_{e_j - r} \varphi(r))},
%\end{equation}
%where $e_1, \ldots, e_n$ are the vertices of $\Delta_n$, $\nabla_{e_i - r}$ is the directional derivative, and $r$ is the unique element in $\Delta_n$ with $r_i \propto 1/p_i$ (written $r = p^{-1}$).  This extends the classical convex duality and gives rise naturally to the {\it $L$-divergence} defined by
%\begin{equation} \label{eqn:L.divergence}
%{\bf D}\left[ r : r' \right] = \log \left(1 + \nabla \varphi(r') \cdot (r - r')\right) - \left( \varphi(r) - \varphi(r')\right), \quad r, r' \in \Delta_n.
%\end{equation}
Very roughly, given $p$ in the support of $P$, the image $q = T(p)$ under the Monge solution is given as follows. Let $r$ be the unique element in $\Delta_n$ such that $r_i \propto 1/p_i$ for all $i\in [n]:=\{ 1,2,\ldots, n \}$. Also, let $\pi$ denote the unique element in $\Delta_n$ such that $\pi_i \propto q_i / p_i$ for each $i \in [n]$. Then 
\[
\pi_i = r_i (1 + \nabla_{e_i - r} \varphi(r)), \quad \forall\; i\in [n],
\] 
where $e_1, \ldots, e_n$ are the vertices of $\Delta_n$ and $\nabla_{e_i - r}$ is the directional derivative. The map $r\in \Delta_n \rightarrow \pi =: \boldsymbol{\pi}(r)\in \Delta_n$ is called the \textit{portfolio map} generated by $\varphi$ due to its first appearance in stochastic portfolio theory \cite{F02, PW14}. The identity transport $T(p) \equiv p$ corresponds to the exponentially concave function $\varphi_0(r) := \frac{1}{n} \sum_{i=1}^n \log r_i$ and the induced portfolio map $\boldsymbol{\pi}(r)\equiv (1/n, 1/n, \ldots, 1/n)$ is called the \textit{equal-weighted portfolio}. 

It is helpful to think of this transport problem as a multiplicative analogue of the well studied transport problem on $\rr^n$ with cost $c(x,y)=\norm{x-y}^2$. The map $x \mapsto -x$ is a group operation on $\rr^n$. In our case it is the map $p \mapsto r := p^{-1}$.  The difference $y-x$, between $x\in \rr^n$ and its optimal Monge image $y$, is replaced by the portfolio $\pi$. This multiplicative theme permeates all our arguments. However, the transport cost is no longer the squared Euclidean norm, but the relative entropy (see Lemma \ref{lem:cost.as.entropy}). Hence, the transport cost is asymmetric and not a metric between probability measures.

One can think of the Wasserstein transport as being performed by adding (conditioned) Gaussian increments with vanishingly small noise. This is essentially the Schr\"odinger problem approach to optimal transport due to L\'eonard \cite{L12, leonardsurvey}. See also the related stochastic analysis of Schr\"odinger bridges in \cite{conforti2018}. In Section \ref{sec:MSP} we study the analogue for our transport problem. We show in Theorem \ref{thm:static.transport} that our transport corresponds to multiplying by gamma random variables with mean going to infinity (and scale one) and suitably normalizing. Let us give an informal description of the statement of this result since we deviate from the usual gamma convergence. Let $p = (p_1, \ldots, p_n) \in \Delta_n$ be given, and let $G=(G_1, \ldots, G_n)$ be a vector of i.i.d. gamma random variables with mean $\lambda/n >0$ and scale one. Define the $\Delta_n$-valued random vector $Q = (Q_1, \ldots, Q_n)$ where
\begin{equation} \label{eqn:p.dot.G0}
Q_i = \frac{p_i G_i}{\sum_{j = 1}^n p_j G_j}, \quad \forall \; i \in [n].
\end{equation}
Alternatively, we can replace $G$ by $D = (D_1, \ldots, D_n)$ where $D_i = \frac{G_i}{\sum_{j = 1}^n G_j}$. Thus $D$ has the Dirichlet distribution with parameters $(\lambda/n, \ldots, \lambda/n)$. {Note that as $\lambda \rightarrow \infty$ the Dirichlet random vector $D$ concentrates at $(1/n, \ldots, 1/n)$, i.e., the multiplicative noise in \eqref{eqn:p.dot.G0} tends to zero.}

Fix $P_0$ and $P_1$, two absolutely continuous probability distributions on $\Delta_n$. Sample two independent i.i.d. sequence $\{ p(j),\; j \in \NN\}$ from $P_0$ and $\{ q(j),\; j \in \NN \}$ from $P_1$. Consider a positive sequence $\{\lambda_N,\; N \in \NN\}$. For every $N$, and for each $j\in [N]$, generate an independent vector of gamma random variables $G(j)$ with mean $\lambda_N/n$. Multiply $G(j)$ with $p(j)$ as in \eqref{eqn:p.dot.G0} and construct a sequence $Q(j)=\left( Q_1(j), \ldots, Q_n(j)\right)\in \Delta_n$. Note that both $G$ and $Q$ depend on $N$. Now \textit{condition} on the event that the following two empirical distributions coincide:
\[
\frac{1}{N} \sum_{j=1}^N \delta_{Q(j)} = \frac{1}{N} \sum_{j=1}^N \delta_{q(j)}. 
\]
Of course, this is a zero probability event, but it is not hard to {make this intuition precise (see Remark \ref{rmk:discrete.Schrodinger})}. By matching the atoms this leads to an explicit coupling between the two empirical distributions 
\[
L_N(0) := \frac{1}{N} \sum_{j = 1}^N \delta_{p(j)}, \quad L_N(1) := \frac{1}{N} \sum_{j = 1}^N \delta_{q(j)}.
\]
In Theorem \ref{thm:static.transport} we prove the following. Under suitable regularity conditions, if we choose $\lambda_N$ to be of order $N^{2/n}$, then, as $N \rightarrow \infty$, this explicit coupling converges to the optimal Monge coupling between $P_0$ and $P_1$ at a rate $O\left(N^{-1/2n}\sqrt{\log N}\right)$ in the Wasserstein-$2$ metric. The main idea is that the conditional coupling solves the discrete Schr\"odinger problem and can be directly analyzed instead of first taking $N\rightarrow \infty$ and using Sanov's Theorem as done in \cite{ADPZ11, DLR13, EMR15}. This method is robust and extends to other cost functions (such as the quadratic cost) whenever suitable random variables (e.g., Gaussians) can be identified.

In \cite{PW16} we also defined a displacement interpolation that corresponds to linear interpolation between the generating functions $\varphi$ and $\varphi_0$, or, equivalently, between $\boldsymbol{\pi}$ and the equal-weighted portfolio. We showed that each individual particle travels along a straight line in the unit simplex, but the speed is non-uniform and depends on the position. Hence, the displacement interpolation is not McCann's interpolation \cite{M97} where each particle travels at constant velocity. In Theorem \ref{thm:dynconv} we show that our displacement interpolation corresponds to large deviations of the Dirichlet process whose \textit{marginal distribution} is the Dirichlet distribution that is used in the static transport described above. This is analogous to the Wasserstein-$2$ picture where the static Gaussian distribution extends to the dynamic Brownian motion. %It is possible that this particle system may lead to a Markov Chain Monte Carlo algorithm for approximating the optimal transport.

More interestingly, the Lagrangian action corresponding to this dynamics has a natural infinite-dimensional limit. Think of the unit simplex as the set of probability measures with $n$ atoms. This can be seen as a projection (made rigorous in Section \ref{sec:dynamic}) from the set of all Borel probability measures on $[0,1]$. For any such probability measure $\mu$, consider the relative entropy $H(\mathrm{Leb} \mid \mu)$ of the Lebesgue measure (or, uniform distribution) on $[0,1]$ with respect to $\mu$. Our Lagrangian on $\Delta_n$ is this relative entropy functional \textit{passed through the projection} (see Lemma \ref{lem:entropy.restriction} and Definition \ref{def:Lagrangian}).  Another way to express this relative entropy is to consider the distribution function $F$ of $\mu$. Then $H(\mathrm{Leb} \mid \mu)$ is the entropy of the pushforward of $\mathrm{Leb}$ by $F$, an observation taken from the work \cite{vRS09} by von Renesse and Sturm on the entropic measure and Wasserstein diffusion. In particular, our relative entropy Lagrangian appears as the Hamiltonian of the entropic measure $\mathbb{P}^\beta$ in \cite[eqn. (1.1)]{vRS09}. This is a connection that we do not fully understand although the Dirichlet processes are also critical in their construction. 

Next we establish in Section \ref{sec:entropy} the semiconvexity of entropy along the displacement interpolation paths given above. The \textit{reference measure} on the unit simplex with respect to which (relative) entropy is calculated is the Dirichlet distribution with all zero parameters. This is a $\sigma$-finite measure on $\Delta_n$ that is related to the finite-dimensional distributions of the entropic measure (see \cite[Lemma 3.1]{vRS09}). In Theorem \ref{thm:entropy} we prove that if $P_t$, $ t \in [0,1]$, is the displacement interpolation transporting absolutely continuous probability measures $P_0$ to $P_1$ on $\Delta_n$, then the (relative) entropy of $P_t$ (with respect to the reference measure) plus $n$ times the cost of transporting $P_0$ to $P_t$ is convex in $t$. This is highly reminiscent of the semiconvexity of entropy along interpolating lines of Wasserstein-$2$ transport in negative curvature spaces as established in \cite{CMS06}, {and we refer the reader to Remark \ref{rmk:convexity.discussion} for more discussion.} It might also be related to the constant sectional curvature $-1$ of the unit simplex under the dualistic geometry (in the sense of information geometry, see \cite{A16}) induced by an exponentially concave function (see \cite[Cor 4.10]{PW16} and \cite{W17}). Part of the argument involves a new Monge-Amp\`ere equation (Theorem \ref{thm:Monge.Ampere}) which might be of independent interest. Finally, in Section \ref{sec:bounds} we prove a Talagrand-type dimension-free bound on the transportation cost  (Theorem \ref{thm:second.bound}) whose intuition relies on the infinite-dimensional extension of the Lagrangian that is described above.

The motivation for these results stems from our belief that it is possible to develop information geometry \cite{A16} on spaces of probability distributions given by a cost of transport that is not a metric but a divergence in a suitable sense. In particular, Otto calculus  \cite{JKO98, O01} and the theory of gradient flows \cite{AGS08} might have extensions in this non-Riemannian information geometry. Also see \cite{WY19} for a recently discovered differential geometric connection between information geometry and optimal transport.  More broadly, we aim to extend results that go beyond the classical Wasserstein geometry and explore their mathematical implications and potential applications to, for example, statistics and data science. The generator of the gamma subordinator is a non-local operator. So, it is unlikely that the usual Otto calculus extends directly to this context. However, a probabilistic version of gradient flow for this cost function has been proved in \cite[Theorem 2]{Pal19}.

\section{The transport problem} \label{sec:transport.problem}
In this section we gather and prove basic results about our transport problem that are needed in this paper. For more details and motivations the reader may refer to \cite{PW14, PW16}. The proofs for this section are deferred till the Appendix.

\subsection{The cost function}
Let $n \geq 2$ be an integer and consider the open simplex $\Delta_n$ defined in \eqref{eqn:simplex}. Its closure in $\mathbb{R}^n$ is denoted by $\overline{\Delta}_n$.  Any vector in $(0, \infty)^n$ can be normalized to give an element of $\Delta_n$. This leads naturally to the (commutative) group operation
\begin{equation} \label{eqn:simplex.group.operation}
p \odot q := \left( \frac{p_iq_i}{\sum_{j = 1}^n p_jq_j} \right)_{1 \leq i \leq n}, \quad p, q \in \Delta_n.
\end{equation}
The identity element is the barycenter $\overline{e} := \left( \frac{1}{n}, \ldots, \frac{1}{n}\right)$, and the inverse of $p \in \Delta_n$ is given by
\begin{equation} \label{eqn:simplex.inverse}
p^{-1} := \left( \frac{1/p_i}{\sum_{j = 1}^n 1/p_j}\right)_{1 \leq i \leq n}. 
\end{equation}
In fact, $\Delta_n$ is a vector space with the following definition of scalar multiplication: for $\lambda \in \rr$ and $p \in \Delta_n$, let $\lambda \otimes p \in \Delta_n$ be such that 
\[
(\lambda \otimes p)_i = \frac{p_i^\lambda}{\sum_{j=1}^n p_j^\lambda}, \quad i\in [n].
\] 
We endow $\Delta_n$ with the topology which is consistent with the vector space structure. In compositional data analysis \eqref{eqn:simplex.group.operation} and \eqref{eqn:simplex.inverse} are called the perturbation and powering operations respectively, and are used to define the Aitchison geometry on the simplex \cite{EP06}.

\medskip

We also introduce a $\sigma$-finite reference measure on the unit simplex, namely the Dirichlet distribution with all parameters equal to zero. For a given $n$ it is closely related to the finite-dimensional marginals of the {\it entropic measure} constructed in \cite[Section 3.1]{vRS09}. It will be used in Section \ref{sec:entropy} to study the behavior of entropy.

\begin{defn} [Reference measure] \label{def:entropic.measure}
	We let $\mu_0$ be the $\sigma$-finite measure on $\Delta_n$ defined using the parameterization \eqref{eqn:simplex.parameterization} below by
	\begin{equation} \label{eqn:entropic.measure}
	d\mu_0(p) = \frac{1}{p_1 p_2 \cdots p_{n-1} p_n} d p_1 dp_2 \cdots dp_{n-1}, \quad p \in \mathcal{D}_{n-1},
	\end{equation}
	where $p_n = 1 - p_1 - \cdots - p_{n-1}$ and
	\begin{equation} \label{eqn:simplex.parameterization}
	\mathcal{D}_{n-1} := \{(p_1, \ldots, p_{n-1}) \in \mathbb{R}^{n-1}: p_i > 0, \ p_1 + \cdots + p_{n-1} < 1\}.
	\end{equation}
	It can be verified that $\mu_0$ is the Haar measure on the unit simplex with respect to the commutative group operation $\odot$.
\end{defn}

Throughout this paper we let $c: \Delta_n \times \Delta_n \rightarrow [0, \infty)$ be the cost function defined by \eqref{eqn:cost.function}.
By Jensen's inequality we have $c(p, q) \geq 0$ for all $p, q$, and $c(p, q) = 0$ only if $p = q$. It is clear that the cost function is not symmetric in $p$ and $q$. The asymmetry is captured by the inversion \eqref{eqn:simplex.inverse}. By a straightforward computation, we have

\begin{lemma} \label{lem:cost.asymmetry}
	For any $p, q \in \Delta_n$ we have $c(q, p) = c(p^{-1}, q^{-1})$.
\end{lemma}

\begin{rmk} \label{rmk:dual.simplex}
	The inversion $p \leftrightarrow p^{-1}$ sets up a duality between two copies of the simplex. Motivated by the information-geometric results and terminologies of \cite[Section 3]{PW16} we regard $p, q \in \Delta_n$ in \eqref{eqn:cost.function} as elements of the {\it dual} simplex, and $p^{-1}, q^{-1}$ as elements of the {\it primal} simplex. We call $c^*(p, q) = c(p^{-1}, q^{-1}) = c(q, p)$ the {\it dual} cost function. In this paper we focus on the dual simplex and the word dual is omitted. When the duality is important it will be made explicit, such as in Section \ref{sec:convergence}.
\end{rmk}

The following lemma gives an interesting alternative expression in terms of the group operation \eqref{eqn:simplex.group.operation}. The proof is left to the reader. 

\begin{lemma} \label{lem:cost.as.entropy}
	Given $p, q \in \Delta_n$, define $\pi = q \odot p^{-1} \in \Delta_n$. Then
	\begin{equation} \label{eqn:cost.as.entropy}
	c(p, q) = H\left( \overline{e} \mid \pi\right),
	\end{equation}
	where $H$ is the relative entropy defined on $\Delta_n \times \Delta_n$ by
	\begin{equation} \label{eqn:relative.entropy}
	H\left(p \mid q\right) := \sum_{i = 1}^n p_i \log \frac{p_i}{q_i}.
	\end{equation}
\end{lemma}

The variable $\pi$ plays an important role throughout this paper. Following our previous works \cite{PW14, PW16, P17} we call $\pi$ the {\it portfolio vector}. Note that $p = q$ (i.e., $c(p, q) = 0$) if and only if the portfolio vector $\pi$ is equal to the barycenter $\overline{e}$. Since $H(\overline{e} | \cdot)$ is a convex function on the vector space $(\Delta_n, \odot, \otimes)$, our cost function $H(\overline{e}\mid q \odot p^{-1})$ is analogous to the cost $h(y-x)$, for $x,y\in \rr^n$, where $h$ is a convex function on $\rr^n$ (see Remark \ref{rmk:convex.cost} for more discussion).

\begin{defn}[Optimal transport cost]
	Given Borel probability measures $P, Q$ on $\Delta_n$ (written $P, Q \in \mathcal{P}(\Delta_n)$), consider the Monge-Kantorovich optimal transport problem with cost $c$. We define the optimal transport cost by
	\begin{equation} \label{eqn:transport.cost}
	\mathbf{C}(P, Q) := \inf_{R \in \Pi(P, Q)} \mathbb{E}_{(p, q) \sim R} \left[ c(p, q) \right],
	\end{equation}
	where $\Pi(P, Q)$ is the set of couplings of $P$ and $Q$.
\end{defn}

It is clear that $\cost(P, Q)$ is not a metric since it is asymmetric in $P$ and $Q$. Using the language of information geometry \cite{A16} we call $\cost$ a {\it divergence}, of which the relative entropy (also called the Kullback-Leibler divergence) is a classical example.

At several places in this paper we will also make use of the Wasserstein-2 distance defined for Borel probability measures on $\mathbb{R}^d$ by
\begin{equation} \label{eqn:Wasserstein.distance}
\mathcal{W}_2(P, Q) := \inf_{R \in \Pi(P, Q)} \left( \mathbb{E}_{(p, q) \sim R} \left[ \|p - q\|^2 \right] \right)^{1/2},
\end{equation}
where $\|\cdot \|$ is the Euclidean distance.

\subsection{Solution via exponentially concave functions} \label{sec:transport.solution}
In this subsection we describe the solution to our transport problem. First we introduce a subspace of $\mathcal{P}(\Delta_n)$ which will play the role of the classical Wasserstein space $\mathcal{W}_2(\mathbb{R}^d)$.

\begin{defn} [The classes $\mathcal{L}$ and $\mathcal{L}_a$]
	Let $P \in \mathcal{P}(\Delta_n)$ be a Borel probability measure on $\Delta_n$. We let $\mathcal{L}$ be the set of all $P \in \mathcal{P}(\Delta_n)$ such that
	\begin{equation} \label{eqn:log.moment.condition}
	\sum_{i = 1}^n \int_{\Delta_n} \left| \log p_i \right| dP(p)< \infty.
	\end{equation}
	We let $\mathcal{L}_a$ be the subset consisting of probability measures in $\mathcal{L}$ that are absolutely continuous with respect to the $(n-1)$-dimensional Lebesgue measure on $\Delta_n$.
\end{defn}

Note that if $P \in \mathcal{L}$, then the pushforward of $P$ under the inversion map $p \mapsto p^{-1}$ also belongs to $\mathcal{L}$. The same is true for $\mathcal{L}_a$. 

\begin{lemma} \label{prop:class.L}
	Let $P, Q \in \mathcal{L}$. Then for any coupling $R \in \Pi(P, Q)$ we have
	\[
	\int c(p, q) dR(p, q) \leq 2\left( \sum_{i = 1}^n  \int \left| \log p_i \right| dP(p) + \int \left| \log q_i \right| dQ(q)\right) < \infty.
	\]
	In particular, we have $\mathbf{C}(P, Q) < \infty$.
\end{lemma}

%As shown in \cite{PW14} the transport problem can be solved in terms of exponentially concave functions.

\begin{defn} [Exponentially concave function]
	A function $\varphi : \Delta_n \rightarrow \mathbb{R}$ is exponentially concave if $e^{\varphi}$ is concave.
\end{defn}

%Clearly we can define exponentially concave functions on any convex domain. 
Let $\varphi$ be exponentially concave. Since $e^{\varphi}$ is concave, by well known results in convex analysis (see \cite{R70}) it can be shown that $\varphi$ is differentiable Lebesgue almost everywhere. In particular, its gradient $\nabla \varphi$ is a.e.~defined on $\Delta_n$. 

\begin{defn}[Portfolio map]
	Let $\varphi$ be exponentially concave on $\Delta_n$. When $\varphi$ is differentiable at $r \in \Delta_n$, we define $\boldsymbol{\pi}(r) \in \overline{\Delta}_n$ by
	\begin{equation} \label{eqn:portfolio.map}
	(\boldsymbol{\pi}(r))_i = r_i \left(1 + \nabla_{e_i - r} \varphi(r)\right), \quad i = 1, \ldots, n,
	\end{equation}
	where $\{e_1, \ldots, e_n\}$ is the standard basis of $\mathbb{R}^n$ and $\nabla_{e_i - r}$ is the directional derivative. We call $\boldsymbol{\pi}$ the portfolio map generated by $\varphi$.
\end{defn}

An important property of the portfolio map is multiplicative cyclical monotonicity (see \cite[Proposition 4]{PW14}): if $m \geq 1$ and $r(0), r(1), \ldots, r(m) = r(0)$ is a cycle in $\Delta_n$, then
\begin{equation} \label{eqn:MCM}
\prod_{s = 0}^{m - 1} \left(\sum_{i = 1}^n (\boldsymbol{\pi}(r(s)))_i \frac{r_i(s + 1)}{r_i(s)}\right) \geq 1.
\end{equation}
In \cite{PW14} we showed that this condition characterizes $c$-cyclical monotonicity of our transport problem. The following result can be viewed as the analog of Brenier's theorem \cite{B91} in our context. Its proof is given in the Appendix.    

\begin{thm} \label{thm:transport.solution}
	Consider the optimal transport problem \eqref{eqn:transport.cost}. If $P \in \mathcal{L}_a$ and $Q \in \mathcal{L}$, then there exists an exponentially concave function $\varphi: \Delta_n \rightarrow \mathbb{R}$ such that the following statements hold.\footnote{Although we use the same notation $\Delta_n$, it is helpful to regard $\varphi$ as a function on the {\it primal} simplex. See Remark \ref{rmk:dual.simplex} and compare with \eqref{eqn:deterministic.transport}.}
	\begin{itemize}
		\item[(i)] If $\boldsymbol{\pi}$ is the portfolio map generated by $\varphi$, the mapping
		\begin{equation} \label{eqn:deterministic.transport}
		p \mapsto q = T(p) := p \odot \boldsymbol{\pi}(p^{-1}),
		\end{equation}
		which is $P$-a.e.~defined, pushforwards $P$ to $Q$.
		\item[(ii)] The deterministic coupling $(p, T(p))$ is optimal for the transport problem \eqref{eqn:transport.cost}, and is $P$-a.e.~unique.
	\end{itemize}
\end{thm}

\begin{rmk} \label{rmk:convex.cost}
Our cost function \eqref{eqn:cost.function} can be expressed as a convex cost of the form $h(\theta - \phi)$ by passing to the exponential coordinate system (see the proof of Lemma \ref{prop:class.L} and \cite{PW16}). While this allows us to use the twist condition to obtain a formula of the Monge solution (see \cite[Theorem 1.17]{S15} and \cite{GM96}), the portfolio map does not occur naturally there. However, consideration of the portfolio map is crucial in our approach, especially the displacement interpolation. Our argument also has an intuitive and financial flavor as it was motivated by stochastic portfolio theory. Thus the portfolio map is an additional structure not shared by a generic convex cost. Generalizations of the portfolio map, for cost functions defined by cumulant generating functions, can be found in \cite{P17}.
\end{rmk}

Note that if we write $r = p^{-1}$, then we can write \eqref{eqn:deterministic.transport} in the form 
\begin{equation} \label{eqn:q.as.weight.ratio}
q_i = \frac{(\boldsymbol{\pi}(r))_i/r_i}{ \sum_{j = 1}^n (\boldsymbol{\pi}(r))_j/r_j}, \quad i \in [n].
\end{equation}

\subsection{$L$-divergence}
Apart from the portfolio map, an exponentially concave function on $\Delta_n$ defines another fundamental quantity called the $L$-divergence. It can be regarded as a distance-like quantity on the simplex induced by the transport map. For simplicity and to focus on the main ideas, we will impose regularity conditions on $\varphi$ whenever needed. In-depth studies of the $L$-divergence and its generalizations can be found in \cite{W15, PW16, W17, W19, WY19}. 

\begin{defn} [$L$-divergence] \label{def:L.divergence}
	Let $\varphi$ be a differentiable exponentially concave function on $\Delta_n$. The $L$-divergence of $\varphi$ is defined by
	\begin{equation} \label{eqn:L.divergence}
	{\bf D}\left[ r : r' \right] = \log \left(1 + \nabla \varphi(r') \cdot (r - r')\right) - \left( \varphi(r) - \varphi(r')\right), \quad r, r' \in \Delta_n,
	\end{equation}
	where $\nabla$ is the Euclidean gradient and $a \cdot b$ is the Euclidean dot product.
\end{defn}

By the exponential concavity of $\varphi$, it can be shown that ${\bf D}\left[ r : r' \right] \geq 0$ and ${\bf D}\left[ r : r \right] = 0$. If $e^{\varphi}$ is strictly concave, then ${\bf D}\left[ r : r' \right] > 0$ for all $r \neq r'$. Using the definition of $\boldsymbol{\pi}$ (see \eqref{eqn:portfolio.map}), we can write
\begin{equation}  \label{eqn:L.divergence.portfolio}
{\bf D}\left[ r : r' \right] = \log \left( \sum_{i = 1}^n (\boldsymbol{\pi}(r'))_i \frac{r_i}{r_i'} \right) - \left( \varphi(r) - \varphi(r')\right).
\end{equation}

\begin{exm}
	Suppose in Theorem \ref{thm:transport.solution} we let $P = Q$. Since $c(p, q) \geq 0$ and equality holds only if $p = q$, the optimal coupling is the identity $q = T(p) \equiv p$. This is induced by the distinguished exponentially concave function
	\begin{equation} \label{eqn:log.geometric.mean}
	\varphi_0(r) := \frac{1}{n}\sum_{i = 1}^n  \log r_i, \quad r \in \Delta_n.
	\end{equation}
	To see this, note that the portfolio map generated by \eqref{eqn:log.geometric.mean} is the constant map
	\begin{equation} \label{eqn:equal.weighted.portfolio}
	\boldsymbol{\pi}(r) \equiv \overline{e} = \left(\frac{1}{n}, \ldots, \frac{1}{n}\right), \quad r \in \Delta_n.
	\end{equation}
	By \eqref{eqn:deterministic.transport}, we have $q = p \odot \boldsymbol{\pi}(p^{-1}) = p \odot \overline{e} = p$.
	The induced $L$-divergence is our cost function $c$, i.e.,
	\begin{equation} \label{eqn:L.divergence.cost}
	{\bf D}\left[ r : r' \right] = c(r', r) = \log \left( \frac{1}{n}\sum_{i = 1}^n  \frac{r_i}{r_i'}\right) - \frac{1}{n}\sum_{i = 1}^n  \log \frac{r_i}{r_i'}.
	\end{equation}
\end{exm}

\medskip

Now suppose that $\varphi$ is twice differentiable. Let $\nabla^2 \varphi$ be the Euclidean Hessian of $\varphi$. Then, exponential concavity of $\varphi$ is equivalent to the condition
\begin{equation} \label{eqn:matrix.L}
L(r) := -\nabla^2 \varphi(r) - \nabla \varphi(r) \otimes \nabla \varphi(r) \geq 0, \quad \text{for all $r\in \Delta_n$},
\end{equation}
as a quadratic form. Regarding $L(r)$ as an $n \times n$ matrix, for any (column) tangent vector $v  \in \mathbb{R}^n$ with $v_1 + \cdots v_n = 0$ we have
\begin{equation}  \label{eqn:matrix.L.inequality}
v^{\top} L(r) v = - e^{-\varphi(r)} \left. \frac{d^2}{dh^2} e^{\varphi(r + hv)} \right|_{h = 0} \geq 0.
\end{equation}
This also gives the estimate
\begin{equation} \label{eqn:L.divergence.approx}
{\bf D}[r + hv : r] = \frac{h^2}{2} v^{\top} L(r) v + O(|h|^3), \quad  h \rightarrow 0.
\end{equation}

We will make use of the following result which was proved in \cite[Theorem 3.2]{PW16} using an exponential coordinate system. Also see Lemma \ref{lem:det.Jt} below which computes the Jacobian of the transport map.

\begin{lemma} \label{lem:diffeo}
	Let $\varphi: \Delta_n \rightarrow \mathbb{R}$ be $C^2$ and exponentially concave, and let $\boldsymbol{\pi}$ be the portfolio map generated by $\varphi$. If $L(r)$ is positive definite in the sense that $v^{\top} L(r) v > 0$ for all nonzero tangent vectors $v$ and all $r \in \Delta_n$, then the transport map $T(p) := p \odot \boldsymbol{\pi}(p^{-1})$ is a $C^1$-diffeomorphism from $\Delta_n$ onto its image.
\end{lemma}

\subsection{Displacement interpolation} \label{sec:interpolation}
Let $P_0, P_1 \in \mathcal{L}_a$. By Theorem \ref{thm:transport.solution} there exists an exponentially concave function $\varphi_1$ on $\Delta_n$ such that the deterministic transport
\begin{equation} \label{eqn:transport.at.time.1}
p \mapsto q = T_1(p) := p \odot \boldsymbol{\pi}_1(p^{-1}),
\end{equation}
where $\boldsymbol{\pi}_1$ is the portfolio map generated by $\varphi_1$, is the a.e.~unique solution of the transport problem for the pair $(P_0, P_1)$.

\begin{notn} \label{not:portfolio.maps}
	If $\boldsymbol{\pi}$ is a portfolio map, we use $(\boldsymbol{\pi})_i(\cdot)$ (but not $\boldsymbol{\pi}_i(\cdot)$) to denote its $i$-th component. Thus the symbol $\boldsymbol{\pi}_1$ in \eqref{eqn:transport.at.time.1} means the transport map ``at time $1$" but not its components which are denoted by $(\boldsymbol{\pi}_1)_i$. For a fixed element $\pi \in \Delta_n$ (without bold font) we denote its components by $\pi = (\pi_i)$.
\end{notn}

Recall the exponentially concave function $\varphi_0$ defined by \eqref{eqn:log.geometric.mean}. Using the inequality of the arithmetic and geometric means, it is easy to see that the function
\begin{equation} \label{eqn:varphi.t}
\varphi_t := (1 - t) \varphi_0 + t \varphi_1
\end{equation}
is exponentially concave for $0 \leq t \leq 1$. From \eqref{eqn:portfolio.map}, it generates a portfolio map $\boldsymbol{\pi}_t$ which is a linear interpolation between the equal-weighted portfolio and $\boldsymbol{\pi}_1$:
\begin{equation} \label{eqn:pi.t}
\boldsymbol{\pi}_t = (1 - t) \overline{e} + t \boldsymbol{\pi}_1
\end{equation}
This leads to the definition, taken from \cite{PW16}, of displacement interpolation for our transport problem.

\begin{defn}[Displacement interpolation] \label{def:displacement.interpolation}
	Let $P_0, P_1 \in \mathcal{L}_a$. For $0 \leq t \leq 1$, let $T_t$ be the map defined $P_0$-a.e.~by
	\begin{equation} \label{eqn:transport.time.t}
	T_t(p) = p \odot \boldsymbol{\pi}_t(p^{-1}),
	\end{equation}
	where $\boldsymbol{\pi}_t$ is given by \eqref{eqn:pi.t}. We define the displacement interpolation $\{P_t\}_{0 \leq t \leq 1}$ by
	\begin{equation} \label{eqn:displacement.interpolation}
	P_t := (T_t)_{\#} P_0.
	\end{equation}
\end{defn}

\begin{rmk} \label{rmk:straight.line}
	We emphasize that our displacement interpolation is fundamentally different from the one defined by McCann \cite{M97} for the quadratic cost $\|p - q\|^2$ on $\mathbb{R}^d$. If we let $[P, P']_t$ denote McCann's displacement interpolation for the measures $P$ and $P'$, then each individual particle travels along a constant velocity straight line and the following properties hold: (i) (time symmetry) $[P, P']_t = [P', P]_{1 - t}$ and (ii) (time consistency) $[[P, P']_t, [P, P']_{t'}]_s = [P, P']_{(1 - s) t + st'}$. Simple examples show that both properties fail for our interpolation \eqref{eqn:displacement.interpolation}. However, our interpolation has the intermediate optimality  property (for each pair $(P_0, P_t)$) that the McCann interpolation in this case does not. In Section \ref{sec:dynamic} we will relate our interpolation with a Lagrangian action. From \eqref{eqn:transport.time.t} (also see \cite{PW16}) it follows that for $p$ fixed, the path $\{p \odot \boldsymbol{\pi}_t(p^{-1})\}_{0 \leq t \leq 1}$ is a straight line in the (dual) simplex $\Delta_n$ run at non-uniform speed. 
\end{rmk}

In order that our displacement interpolation makes sense we need to show that $P_t \in \mathcal{L}_a$ for each $t$. This is accomplished in the next proposition (c.f. \cite[Proposition 5.19(iii)]{V03}) whose proof can be found in the Appendix. 

\begin{prop} \label{prop:interpolation}
	For $P_0, P_1 \in \mathcal{L}_a$, we have $P_t \in \mathcal{L}_a$ for each interpolant of the displacement interpolation $\{P_t\}_{0 \leq t \leq 1}$. If $P_0 \neq P_1$, the transport cost $\mathbf{C}(P_0, P_t)$ is smooth, increasing and strictly convex in $t$.
\end{prop}

This definition of displacement interpolation is tailored for our cost function. In fact, we will show in Section \ref{sec:entropy} that under suitable technical conditions on $P_0$ and $P_1$ the map $t \mapsto \mathrm{Ent}_{\mu_0} (P_t) + n \mathbf{C}(P_0, P_t)$ is convex, where $\mathrm{Ent}_{\mu_0} (P_t)$ is the entropy of $P_t$ with respect to the reference measure $\mu_0$.

\section{Multiplicative Schr\"{o}dinger problem} \label{sec:MSP}
In this section we present a probabilistic solution to the transport problem in terms of an independent particle system driven by Dirichlet processes. We first tackle the static transport problem \eqref{eqn:transport.cost} and then formulate and prove a dynamic version that is consistent with our displacement interpolation.

\subsection{The Dirichlet transport}
%We begin by identifying the stochastic process corresponding to our transport problem motivated by \cite{L12} and \cite{ADPZ11}. 
Consider the gamma distribution, a two-parameter family $\{\mathrm{Gamma}(\alpha, \beta): \alpha > 0, \beta > 0\}$ of probability distributions on $(0, \infty)$. The density function is given by
\[
\frac{\beta^{\alpha}}{\Gamma(\alpha)} y^{\alpha - 1} e^{-\beta y}, \quad y > 0,
\]
where $\Gamma(\cdot)$ is the gamma function. We write $\mathrm{Gamma}(\alpha) = \mathrm{Gamma}(\alpha, 1)$.

Let $p = (p_1, \ldots, p_n) \in \Delta_n$ be given, and let $G_1, \ldots, G_n$ be independent such that $G_i \sim \mathrm{Gamma}(\alpha_i)$, for some constants $\alpha_1, \ldots, \alpha_n > 0$. Define the $\Delta_n$-valued random vector $Q = (Q_1, \ldots, Q_n)$ where
\begin{equation} \label{eqn:p.dot.G}
Q_i = \frac{p_i G_i}{\sum_{j = 1}^n p_j G_j}, \quad 1 \leq i \leq n.
\end{equation}
If we let $D = (D_1, \ldots, D_n)$ where $D_i = \frac{G_i}{\sum_{j = 1}^n G_j}$, then $D$ has the Dirichlet distribution with parameters $(\alpha_1, \ldots, \alpha_n)$. Using the group operation \eqref{eqn:simplex.group.operation} we can write $Q = p \odot D$. Intuitively, we think of $p$ and $Q$ as the positions of a particle at time zero and time one respectively.

Let us find the distribution of $Q$. On the unit simplex $\Delta_n$ we use the Euclidean coordinate system $(p_1, \ldots, p_{n-1})$ where the last component $p_n$ is dropped. The range of $(p_1, \ldots, p_{n-1})$ is the domain $\mathcal{D}_{n-1}$ defined in \eqref{eqn:simplex.parameterization}.

\begin{lemma} \label{lem:Dirichlet.transport.density1}
	For $p \in \Delta_n$ and $\alpha_1, \ldots, \alpha_n > 0$ fixed, the density of $Q$ (or, rather, $(Q_1, \ldots, Q_{n-1})$) with respect to the Lebesgue measure on $\mathcal{D}_{n-1}$ is given by
	\begin{equation} \label{eqn:p.dot.D.density}
	f(q \mid p) = \frac{\Gamma(\sum_{j = 1}^n \alpha_j)}{\prod_{j = 1}^n { q_j} \Gamma(\alpha_j)} \prod_{i = 1}^n \left( \frac{q_i}{p_i}\right)^{{ \alpha_i } } \left(\sum_{i = 1}^n \frac{q_i}{p_i}\right)^{-\sum_{j = 1}^n \alpha_j},
	\end{equation}
	where $q_n:=1- \sum_{i=1}^{n-1} q_i$.
\end{lemma} 
\begin{proof}
	Let $Y_i = p_i G_i$ which is distributed as $\mathrm{Gamma}(\alpha_i, 1/p_i)$. Then the joint density of $(Y_1, \ldots, Y_n)$ on $(0, \infty)^n$ is given by the product
	\[
	\prod_{i = 1}^n \frac{p_i^{-\alpha_i}}{\Gamma(\alpha_i)} y_i^{\alpha_i - 1} e^{-y_i/p_i}.
	\]
	Consider the change of variable $(y_1, \ldots, y_n) \mapsto (q_1, \ldots, q_{n-1}, s)$, where $s = \sum_{j = 1}^n y_j$ and $q_i = y_i/s$. It can be easily verified (matrix determinant lemma) that the Jacobian determinant of this transformation is {$s^{-(n-1)}$}. Also let $q_n:=1- \sum_{i=1}^{n-1} q_{i}$. Then the joint density of $(q_1, \ldots, q_{n-1}, s)$ is given by
	\begin{equation} \label{eqn:density.qs}
	s^{n -1 + \sum_{j = 1}^n (\alpha_j - 1)} e^{-s \sum_{j = 1}^n q_j/p_j} \prod_{i = 1}^n \frac{p_i^{-\alpha_i}}{\Gamma(\alpha_i)} q_i^{\alpha_i - 1}.
	\end{equation}
	Since
	\[
	\int_0^{\infty} s^{n -1 + \sum_{j = 1}^n \alpha_j } e^{-s \sum_{j = 1}^n q_j/p_j} ds = \Gamma\left(\sum_{j = 1}^n \alpha_j\right) \left( \sum_{j = 1}^n \frac{q_j}{p_j} \right)^{-\sum_{j = 1}^n \alpha_j},
	\]
	integrating \eqref{eqn:density.qs} with respect to $s$ gives the result.
\end{proof}

\begin{lemma}  \label{lem:Dirichlet.transport.density2}
	For $\lambda > 0$ let $f_{\lambda}(q \mid p)$ be the density in \eqref{eqn:p.dot.D.density} where $\alpha_i = \frac{\lambda}{n}$ for all $i$. Then
	\begin{equation} \label{eqn:density.qs2}
	f_{\lambda}(q \mid p) = \frac{\Gamma(\lambda)}{ \Gamma(\lambda/n)^n} \frac{1}{\prod_{i = 1}^n {q_i}} \prod_{i = 1}^n \left( \frac{q_i}{p_i}\right)^{ {\frac{\lambda}{n} }} \left( \sum_{i = 1}^n \frac{q_i}{p_i}\right)^{-\lambda}.
	\end{equation}
	Moreover, we have
	\begin{equation} \label{eqn:density.LDP}
	\lim_{\lambda \rightarrow \infty} \frac{-1}{\lambda} \log f_{\lambda}(q \mid p) = c(p, q),
	\end{equation}
	where $c(p, q)$ defined by \eqref{eqn:cost.function} is our cost function, and the convergence holds locally uniformly on $\Delta_n \times \Delta_n$.
\end{lemma}
\begin{proof}
	The formula \eqref{eqn:density.qs2} of the density follows directly from Lemma \ref{lem:Dirichlet.transport.density1}. Now take logarithm, divide by $\lambda$ and take the limit as $\lambda \rightarrow \infty$. It is easy to see that the following limit holds uniformly over compact sets:
	\begin{equation} \label{eqn:density.limit.computation}
	\begin{split}
	\lim_{\lambda \rightarrow \infty} \frac{-1}{\lambda} \log f_T(q \mid p) &= \log \left( \sum_{j = 1}^n \frac{q_j}{p_j}\right) - \frac{1}{n} \sum_{j = 1}^n \log \frac{q_j}{p_j} + \lim_{\lambda \rightarrow \infty} \frac{1}{\lambda} \sum_{i = 1}^n \log {q_i} \\
	&\quad - \lim_{\lambda \rightarrow \infty} \frac{1}{\lambda} \left[ \log \Gamma(\lambda) - n \log \Gamma(\lambda/n)\right]\\
	&= c(p, q) + \log n - \lim_{\lambda \rightarrow \infty} \left[ \log \Gamma(\lambda) - n \log \Gamma(\lambda/n)\right].
	\end{split}
	\end{equation}
	By Stirling's approximation, we have
	\[
	\frac{1}{\lambda} \left[ \log \Gamma(\lambda) - n \log \Gamma\left(\frac{\lambda}{n}\right)\right] = \frac{1}{\lambda} \left[ \lambda \log \lambda - \lambda \log \frac{\lambda}{n}\right] + \frac{O(\log \lambda)}{\lambda}.
	\]
	Hence
	\[
	\lim_{\lambda \rightarrow \infty} \left[ \log \Gamma(\lambda) - n \log \Gamma\left(\frac{\lambda}{n}\right)\right] = \log n
	\]
	and we obtain the desired limit \eqref{eqn:density.LDP}.
\end{proof}

The limit \eqref{eqn:density.LDP} suggests (and it is not hard to prove) that the family of measures corresponding to the densities $\{f_{\lambda}( \cdot \mid p)\}_{\lambda > 0}$ satisfies a large deviations principle (LDP), as $\lambda \rightarrow \infty$, with rate $\lambda$ and a good rate function $c(p, \cdot)$. 

\begin{rmk}
	Note the appearance of the term $1 / \prod_{i = 1}^n q_i$ which is the density of our reference measure $\mu_0$, the Dirichlet distribution with zero parameters. If we let $\tilde{f}_{\lambda}$ be the density of $Q$ with respect to this measure, then $\tilde{f}_{\lambda}(q \mid p) = e^{-\lambda c(p, q) + K_{\lambda}}$ where $K_{\lambda}$ is a normalizing constant. This parallels the quadratic case where one considers the density of $N(x, \sigma^2 I)$ with respect to the Lesbesgue measure and $\sigma^2 \rightarrow 0$. 
\end{rmk}

\subsection{Discrete Schr\"odinger problem: the particle system} \label{sec:particle.system}
Let $P_0, P_1 \in \mathcal{L}_a$. In \cite{L12} C.~L\'{e}onard used gamma convergence to show that the optimal coupling of the Monge-Kantorovich problem can be recovered from the so-called Schr\"{o}dinger problem which minimizes an entropic cost. The specific case of quadratic cost where the solutions can be recovered using Brownian motion was studied much earlier by Mikami \cite{Mikami04} using a stochastic control approach.

While we will keep the spirit,  we deviate from both these approaches and characterize the solution as the limit of explicit couplings constructed from a particle system. In particular, this allows us to avoid the somewhat heavy analytic machinery behind gamma convergence. It will be clear that our methods are robust and can be applied to other cost functions as soon as suitable stochastic processes are identified. 

Given $P_0$ and $P_1$, let $(\Omega, \mathcal{F}, \mathbb{P})$ be a probability space over which the following pair of independent random vectors are defined: $p(1), p(2), \ldots$ are sampled i.i.d.~from $P_0$, and $q(1), q(2), \ldots$ are sampled i.i.d.~from $P_1$. For $N \geq 1$, consider the corresponding empirical measures
\begin{equation} \label{eqn:empirical.measures}
L_N(0) := \frac{1}{N} \sum_{i = 1}^N \delta_{p(i)}, \quad L_N(1) := \frac{1}{N} \sum_{j = 1}^N \delta_{q(j)}
\end{equation}
that are random elements of $\mathcal{P}(\Delta_n)$.

{Let $\lambda > 0$ be given.} Given the realizations $\{p(i)\}$ and $\{q(j)\}$, we construct a coupling $M_N$ of $L_N(0)$ and $L_N(1)$ using the density \eqref{eqn:density.qs2}. Let $\mathcal{S}_N$ be the group of permutations of $N$ labels. For each $\sigma \in \mathcal{S}_N$, let
\begin{equation} \label{eqn:M.N.sigma}
M_N^{\sigma} := \frac{1}{N} \sum_{i = 1}^N \delta_{(p(i), q(\sigma(i)))}.
\end{equation}
We define $M_N$ as a mixture of $\{M_N^{\sigma}\}_{\sigma \in \mathcal{S}_N}$:
\begin{equation} \label{eqn:M.N}
M_N := \sum_{\sigma \in \mathcal{S}_N} \nu_N^{\sigma} M_N^{\sigma},
\end{equation}
where the weight $\nu_N^{\sigma}$ is given by
\begin{equation} \label{eqn:mixture.weight}
\nu_N^{\sigma} := \frac{\prod_{i = 1}^N f_{\lambda}(q(\sigma(i)) \mid p(i))}{\sum_{\rho \in \mathcal{S}_N} \prod_{i = 1}^N f_{\lambda}(q(\rho(i)) \mid p(i))}, \quad \sigma \in \mathcal{S}_N.
\end{equation}
Since each $M_N^{\sigma}$ couples $L_N(0)$ and $L_N(1)$, so does the mixture $M_N$. Our aim is to prove that $M_N$ converges to the optimal coupling $R^*$ of $(P_0, P_1)$ as $N \rightarrow \infty$ and $\lambda = \lambda_N \rightarrow \infty$ at a suitable rate. 

\begin{rmk}[Discrete Schr\"odinger problem] \label{rmk:discrete.Schrodinger}
	Let us relate the coupling \eqref{eqn:mixture.weight} with the Schr\"odinger problem. First we recall the Schr\"odinger bridge problem as in \cite{L12}. Consider distributions $P, Q$ on the unit simplex. Let $R_\lambda$ denote the distribution of random variables $(X,Y)$ where $X \sim P$, and $Y$, given $X=x$, follows the conditional density $f_\lambda(\cdot \mid x)$ from \eqref{eqn:density.qs2}. The Schr\"odinger bridge is defined as the minimizer of the relative entropy $H(R \mid R_\lambda)$ where $R$ runs over all couplings of $(P, Q)$. The definition is meant to capture Schr\"odinger's original idea of the evolution of a particle system with kernel $f_\lambda$ with given initial and terminal configurations. However, if $Q$ is a discrete distribution, the relative entropy of any coupling with respect to $R_\lambda$ is infinite. Here we argue that the coupling \eqref{eqn:mixture.weight} corresponds to the bridge if we first condition on the terminal distribution.
	
	Imagine $N$ particles in $\Delta_n$ with initial configuration $P=\frac{1}{N}\sum_{i=1}^N \delta_{p(i)}$ and terminal configuration $Q=\frac{1}{N} \sum_{j=1}^N \delta_{q(j)}$. Assume, for simplicity, that all $(p(i), 1\le i\le N)$ and $(q(j), 1\le j \le N)$ are distinct. Label the $i$th particle to be the one that is initially at location $p(i)$. Let $\tilde{q}(j)$ denote independent random variables such that $\tilde{q}(j)$ follows the density $f_\lambda(\cdot \mid p(j))$ for each $j$. Let $R_N^\lambda$ denote the probability distribution of the vector $\left(\tilde{q}(j),\; 1\le j \le N \right)$. 
	
	Let $\mathcal{P}(\Delta_n)$ denote the metric space of probabilities on $\Delta_n$ equipped with the metric of weak convergence. Consider the map $T: \left( \Delta_n\right)^N \rightarrow \mathcal{P}(\Delta_n)$ such that $T\left((r(1), r(2), \ldots, r(N) )\right) = N^{-1} \sum_{i=1}^N \delta_N$. $T$ is a continuous map taking a vector to its empirical distribution. By \cite[Theorem 1]{Pollard}, $R_N^\lambda$ has a $(T, \mu_N^\lambda)$ disintegration (i.e., a regular conditional distribution given $T$) where $\mu_N^\lambda$ is the push-forward of $R_N^\lambda$ by the map $T$. It is a standard measure-theoretic verification (e.g., verify over rectangles) that, on the event $N^{-1} \sum_{i=1}^N \delta_{\tilde{q}(j)}=N^{-1} \sum_{i=1}^N \delta_{{q}(j)}$, this disintegration is explicitly by given by $M_N$ as in \eqref{eqn:M.N} and is unique $\mu_N^\lambda$ a.e. In probabilistic language, the regular conditional distribution of the vector $\left(\tilde{q}(j),\; 1\le j \le N \right)$, given its empirical distribution is $N^{-1} \sum_{i=1}^N \delta_{{q}(j)}$, is given by $M_N$. Thus $M_N$ can be defined to the Schr\"odinger bridge in this discrete setting. 
	Since, all ``Monge couplings'' can be represented by permutations where particle $i$ ends up at position $q(\sigma(i))$, for each $i$, for some $\sigma \in \mathcal{S}_N$, the Schr\"odinger bridge $M_N$ can be seen as a mixture of the Monge maps. Moreover, think of $1/\lambda$ as a `noise' parameter. As $\lambda \rightarrow \infty$, for fixed $N$, the noise reduces to zero and the transport becomes progressively closer to the optimal matching between the two $N$-samples. If we let $\lambda, N \rightarrow \infty$ suitably, one expects to recover the optimal Monge coupling for $(P_0, P_1)$ in the limit. This is shown in Theorem \ref{thm:static.transport} below.   
\end{rmk}

\subsection{Convergence to the optimal coupling} \label{sec:convergence}
Our objective is to prove that for an explicit sequence $\{\lambda_N\}_{N \geq 1}$ the sequence of probability measures $M_N$ converges weakly to the optimal coupling $R^*$ with respect to the cost function $c$, $\mathbb{P}$-almost surely. To do this we need some regularity assumptions on the optimal transport map.

Recall by Lemma \ref{lem:cost.asymmetry} that $c(p, q) = c(q^{-1}, p^{-1})$. In the proof it is more convenient to consider the transport from $q^{-1}$ to $p^{-1}$ rather than from $p$ to $q$. Given $P_0, P_1 \in \mathcal{L}_a$, let $\tilde{P}_0$ and $\tilde{P}_1$ be respectively the pushforwards of $P_0$ and $P_1$ under the maps $p \mapsto p^{-1}$ and $q \mapsto q^{-1}$. Since $P_0, P_1 \in \mathcal{L}_a$, so are $\tilde{P}_0$ and $\tilde{P}_1$. By Theorem \ref{thm:transport.solution}, there exists an exponentially concave function $\varphi$ on $\Delta_n$ such that if $\boldsymbol{\pi}$ is the portfolio map generated by $\varphi$, then the map
\begin{equation} \label{eqn:dual.transport}
q^{-1} \mapsto p^{-1} = T^*(q^{-1}) := q^{-1} \odot \boldsymbol{\pi}(q)
\end{equation}
pushforwards $\tilde{P}_1$ to $\tilde{P}_0$ and is the Monge solution (with respect to $c(q^{-1}, p^{-1})$ for the pair $(\tilde{P}_1, \tilde{P}_0)$. 

Consider the $L$-divergence ${\bf D}\left[ \cdot : \cdot \right]$ of $\varphi$ (see \eqref{eqn:L.divergence}). From \eqref{eqn:L.divergence.approx} we know that ${\bf D}\left[ \cdot : \cdot \right]$ is locally quadratic. For technical purposes we will assume that ${\bf D}\left[ \cdot : \cdot \right]$ is equivalent to the squared distance; this will allow us to apply known results about the convergence rates of a sample empirical distribution to the true distribution in the Wasserstein-2 distance in one step of the proof. We believe it is possible to weaken this assumption.

\begin{asmp} \label{ass:regularity.varphi}
	The function $\varphi$ is $C^2$ on $\Delta_n$, and there exist $\alpha, \alpha' > 0$ such that for all $q, q' \in \Delta_n$ we have
	\begin{equation} \label{eqn:L.divergence.bound}
	\alpha \|q' - q\|^2 \leq {\bf D}\left[ q' : q \right] \leq \alpha' \|q' - q\|^2.
	\end{equation}
\end{asmp}

From the lower bound in \eqref{eqn:L.divergence.bound} we have that the quadratic form $L(q)$ (see \eqref{eqn:matrix.L}) is strictly positive definite. Hence, by Lemma \ref{lem:diffeo} (which only uses $L(q) > 0$) the (dual) transport map $T^*$ in \eqref{eqn:dual.transport} is a $C^1$-diffeomorphism. Consider the map $T: \Delta_n \rightarrow \Delta_n$ defined by
\[
T(p) = \left( (T^*)^{-1}(p^{-1}) \right)^{-1}, \quad p \in \Delta_n.
\]
Since $c(p, q) = c(q^{-1}, p^{-1})$, the map $T$ is the Monge solution to the original problem for $(P_0, P_1)$.  (Note here $(T^*)^{-1}$ is the inverse of the map $T^*$ while the other supercript $-1$ refer to the group operation.)

Before stating the main result we give a set of sufficient conditions for Assumption \ref{ass:regularity.varphi} to hold.\footnote{ {As suggested by an anonymous referee, it would be nice to obtain sufficient conditions directly in terms of the distributions $P_0$ and $P_1$. This is an interesting problem (possibly related to analysis of the corresponding Monge-Amp\`{e}re equation studied in Section \ref{sec:Monge.Ampere}) on its own and is left for future research. On the other hand, once the function $\varphi$ is fixed, the transport map $T$ is optimal for any $P_0$ if we set $P_1 = T_{\#} P_0$.}} The proof is given in the Appendix.

\begin{lemma} \label{lem:suff.conditions}
	Suppose there exist constants $C_1, C_2, C_3 > 0$ such that $-\nabla^2 \varphi \leq C_1 I$, $-\nabla^2 e^{\varphi} \geq C_2 I$ and $\| \nabla e^{\varphi} \| \leq C_3$. Then there exists $\alpha, \alpha' > 0$ such that \eqref{eqn:L.divergence.bound} holds.
\end{lemma}

\begin{thm} \label{thm:static.transport}
	Let $P_0, P_1 \in \mathcal{L}_a$, and assume that the function $\varphi$ in \eqref{eqn:dual.transport} satisfies the conditions in Assumption \ref{ass:regularity.varphi}. Let $R^*$ be the optimal Monge coupling for $(P_0, P_1)$. For any $n \geq 2$, let $\lambda_N = \frac{4}{\alpha} N^{2/n}$. Then, $\mathbb{P}$-almost surely, we have
	\[
	\mathcal{W}_2^2(M_N, R^*) = O\left( N^{-1/n} \log N \right), \quad \text{as}\; N \rightarrow \infty.
	\]
\end{thm}

\begin{proof}
	Since the proof is long we will divide it into several steps.
	
	\medskip
	
	{\it Step 1.} Recall that $T$ is the optimal transport map from $P_0$ to $P_1$ and $T^* $ is the dual transport map from $\tilde{P}_1$ to $\tilde{P}_0$. For each particle $p(j)$, let $\hat{q}(j) = T(p(j)) \in \Delta_n$ be the image of $p(j)$ under $T$. For notational simplicity, let us denote $\pi(j) = \boldsymbol{\pi}(\hat{q}(j))$.
	
	Note that \eqref{eqn:dual.transport} implies that $\pi(j) = \hat{q}(j) \odot p(j)^{-1}$. Using this identity, we observe that for any $\tilde{q}(j) \in \Delta_n$ we have
	\begin{equation} \label{eqn:q.tilde.estimate}
	\begin{split}
	& \log \left( \frac{1}{n} \sum_{i = 1}^n \frac{\tilde{q}_i(j)}{p_i(j)}\right) - \log \left( \frac{1}{n} \sum_{i = 1}^n \frac{\hat{q}_i(j)}{p_i(j)}\right) \\
	&= \log \left( \sum_{i = 1}^n \pi_i(j)  \frac{\tilde{q}_i(j)}{\hat{q}_i(j)} \right) \\
	&= {\bf D}\left[ \tilde{q}(j) : \hat{q}(j) \right] + (\varphi(\tilde{q}(j)) - \varphi(\hat{q}(j)) \\
	&\geq \alpha \|\tilde{q}(j) - \hat{q}(j) \|^2 + (\varphi(\tilde{q}(j)) - \varphi(\hat{q}(j)).
	\end{split}
	\end{equation}
	In the above computation, the second equality follows from \eqref{eqn:L.divergence.portfolio} and the last  one follows form \eqref{eqn:L.divergence.bound}.
	
	For each $N$, let $M_N'$ denote the (random) probability measure
	\[
	M_N' = \frac{1}{N} \sum_{j = 1}^N \delta_{(p(j), \hat{q}(j))}.
	\]
	Since the $p(j) \in \Delta_n$'s are i.i.d.~samples from $P_0$ and $\hat{q}(j) = T(p(j)) \in \Delta_n$, $M_N'$ is the empirical measure of $N$ i.i.d.~samples from the optimal coupling $R^*$. Thus it is natural to expect that $\mathcal{W}_2(M_N', R^*) \rightarrow 0$ as $N \rightarrow \infty$. The convergence will be quantified below, and we will show the same for $M_N$ by comparing it with $M_N'$.
	
	\medskip
	
	{\it Step 2.} Fix $N \geq 1$, the number of particles. Since empirical measures do not depend on the labeling of indices, we will relabel $\{q(j), j \in [N]\}$ (that were sampled independently of $\{p(j), j \in [N]\}$) such that the $L^2$-matching distance between the two samples is minimized:
	\begin{equation} \label{eqn:L2.matching}
	\frac{1}{N} \sum_{j = 1}^N \|q(j) - \hat{q}(j)\|^2 = \min_{\sigma \in \mathcal{S}_N} \frac{1}{N} \sum_{j = 1}^N \| q(\sigma(j)) - \hat{q}(j)\|^2 =: W_N \ \text{(say)}.
	\end{equation}
	That is, after the relabeling, the identity permutation $\iota$ attains the minimum in $W_N = \mathcal{W}_2^2(L_N(1), L_N'(1))$.
	
	For $\lambda > 0$ fixed, from the explicit formula of the density $f_{\lambda}$ from \eqref{eqn:density.qs2} and the estimate \eqref{eqn:q.tilde.estimate}, for any $\sigma \neq \iota$ we have
	\begin{equation*}
	\begin{split}
	& \frac{\prod_{j = 1}^N f_{\lambda}(q(\sigma(j)) \mid p(j))}{\prod_{j = 1}^N f_{\lambda}(\hat{q}(j) \mid p(j))} \\
	&= C(q, \hat{q}) \prod_{j = 1}^N \exp \left[ -\lambda \log \left( \sum_{i = 1}^n \frac{q_i(\sigma(j))}{p_i(j)}\right) + \lambda \log \left( \sum_{i = 1}^n \frac{\hat{q}_i(j)}{p_i(j)}\right) \right], 
	\end{split}
	\end{equation*}
	where the constant
	\[
	C(q, \hat{q}):=\left( \prod_{j=1}^N \prod_{i=1}^n \frac{ q_i(j)}{\hat{q}_i(j) } \right)^{\frac{\lambda}{n}-1}
	\]
	does not depend on the permutation $\sigma$. Hence, this term will get canceled from the numerator and the denominator of $\nu_N^\sigma$ in \eqref{eqn:mixture.weight}. Nevertheless, from Step 1 we have
	\begin{equation*}
	\begin{split}
	& \frac{\prod_{j = 1}^N f_{\lambda}(q(\sigma(j)) \mid p(j))}{\prod_{j = 1}^N f_{\lambda}(\hat{q}(j) \mid p(j))} \\
	&\leq C(q, \hat{q})\prod_{j = 1}^N \exp\left[ -\alpha \lambda \| q(\sigma(j))  - \hat{q}(j) \|^2 + \varphi(q(\sigma(j))) - \varphi(\hat{q}(j))\right] \\
	&= C(q, \hat{q}) \exp \left[  \lambda \sum_{j = 1}^N (\varphi(q(j)) - \varphi(\hat{q}(j))) \right] \exp \left[ -\alpha \lambda \sum_{j = 1}^N \| q(\sigma(j)) - \hat{q}(j) \|^2 \right].
	\end{split}
	\end{equation*}
	
	On the other hand, by a similar argument, we can get a lower bound for $\sigma = \iota$:
	\begin{equation*}
	\begin{split}
	& \frac{\prod_{j = 1}^N f_\lambda(q(j)) \mid p(j))}{\prod_{j = 1}^N f_\lambda(\hat{q}(j) \mid p(j))}\\
	&\geq C(q, \hat{q}) \exp \left[  \lambda \sum_{j = 1}^N (\varphi(q(j)) - \varphi(\hat{q}(j))) \right] \exp \left[ -\alpha' \lambda \sum_{j = 1}^N \| q(j) - \hat{q}(j) \|^2  \right] \\
	&= C(q, \hat{q}) \exp \left[  \lambda \sum_{j = 1}^N (\varphi(q(j)) - \varphi(\hat{q}(j))) \right] \exp \left[ -\alpha' \lambda N W_N \right],
	\end{split}
	\end{equation*}
	where $W_N$ is given by \eqref{eqn:L2.matching}. In particular, for any $\sigma \in \mathcal{S}_N$ with $\sigma \neq \iota$ we have the estimate
	\begin{equation} \label{eqn:density.ratio.estimate}
	\begin{split}
	\frac{\prod_{j = 1}^N f_{\lambda}(q(\sigma(j)) \mid p(j))}{\prod_{j = 1}^N f_{\lambda}(q(j) \mid p(j))} &= \frac{\prod_{j = 1}^N f_{\lambda}(q(\sigma(j)) \mid p(j))}{\prod_{j = 1}^N f_{\lambda}(\hat{q}(j) \mid p(j))} \frac{\prod_{j = 1}^N f_{\lambda}(\hat{q}(j)) \mid p(j))}{\prod_{j = 1}^N f_{\lambda}(q(j) \mid p(j))} \\
	&\leq \exp \left[ \alpha' \lambda NW_N - \alpha \lambda \sum_{j = 1}^N \| q(\sigma(j)) - \hat{q}(j) \|^2 \right].
	\end{split}
	\end{equation}
	
	\medskip
	
	{\it Step 3.} Let $\delta_N > 0$ be a sequence, to be chosen later, that converges to $0$ as $N \rightarrow \infty$.
	
	Partition $\mathcal{S}_N$ into two disjoint subsets:
	\begin{equation*}
	\begin{split}
	\mathcal{G}_N &:= \left\{ \sigma \in \mathcal{S}_N : \frac{1}{N} \sum_{j = 1}^N 1\{ \|q(\sigma(j)) - \hat{q}(j)\| > \delta_N \} \leq \frac{\log N}{N^{1/n}}\right\},\\
	  \mathcal{G}_N^c &= \mathcal{S}_N \setminus \mathcal{G}_N.
	\end{split}
	\end{equation*}
	
	Consider $\sigma \in \mathcal{G}_N$ and the probability measures $M_N^{\sigma}$ and $M_N'$ on $\Delta_n \times \Delta_n$. There is a coupling between them that couples the atom $(p(j), q(\sigma(j)))$ of $M_N^{\sigma}$, with the atom $(p(j), \hat{q}(j))$ of $M_N'$, with mass $1/N$. The squared Euclidean distance (in $\mathbb{R}^{2n}$) between these two atoms is exactly $\|q(\sigma(j)) - \hat{q}(j)\|^2$. For each $1 \leq j \leq N$, either $\|q(\sigma(j)) - \hat{q}(j)\| \leq \delta_N$, or $\|q(\sigma(j)) - \hat{q}(j) \| \leq \sqrt{2}$ (diameter of $\Delta_n$). Since $\sigma \in \mathcal{G}_N$, there is only a vanishing fraction of indices that do not satisfy the former bound. Hence, for all $\sigma \in \mathcal{G}_N$ we have
	\[
	\mathcal{W}_2^2 (M_N^{\sigma}, M_N') \leq \delta_N^2 + \frac{2\log N}{N^{1/n}} \rightarrow 0.
	\]
	For $\sigma \notin \mathcal{G}_N$ we have the trivial bound $\mathcal{W}_2(M_N^{\sigma}, M_N') \leq \sqrt{2}$ given by the diameter of the simplex. Since $M_N$ is the mixture of $\{M_N^{\sigma}\}$ with weights $\{\nu_N^{\sigma}\}$, the natural mixture coupling gives
	\begin{equation} \label{eqn:bound.Wasserstein}
	\mathcal{W}_2^2 (M_N, M_N') \leq \delta_N^2 + \frac{2\log N}{N^{1/n}} + 2\sum_{\sigma \in \mathcal{G}_N^c} \nu_N^{\sigma}.
	\end{equation}
	
	Hence, in order to show that $\mathcal{W}_2^2 (M_N, M_N') \rightarrow 0$, it suffices to show that $\sum_{\sigma \in \mathcal{G}_N^c} \nu_N^{\sigma}$ tends to $0$ as $N \rightarrow \infty$.
	
	To this end, note that from \eqref{eqn:mixture.weight}, we have
	\begin{equation} \label{eqn:mixture.weight.estimate}
	\begin{split}
	\sum_{\sigma \in \mathcal{G}_N^c} \nu_N^{\sigma} &= \frac{\sum_{\sigma \in \mathcal{G}_N^c} \prod_{j = 1}^N f_{\lambda} (q(\sigma(j)) \mid p(j))}{\prod_{j = 1}^N f_{\lambda}(q(j) | p(j)) + \sum_{\sigma \neq \iota} \prod_{j = 1}^N f_{\lambda}(q(\sigma(j)) \mid p(j))} \\
	&\leq \sum_{\sigma \in \mathcal{G}_N^c} \frac{\prod_{j = 1}^N f_{\lambda}(q(\sigma(j)) \mid p(j))}{\prod_{j = 1}^N f_{\lambda}(q(j) \mid p(j))} \\
	&\leq \sum_{\sigma \in \mathcal{G}_N^c} \exp\left[ \alpha' \lambda N W_N - \alpha \lambda \sum_{j = 1}^N \| q(\sigma(j)) - \hat{q}(j) \|^2 \right], \quad \text{by \eqref{eqn:density.ratio.estimate}}\\
	&\leq N! \exp\left[ \alpha' \lambda NW_N - \alpha \lambda \frac{N \delta_N^2 \log N}{N^{1/n}} \right].
	\end{split}
	\end{equation}
	The last inequality uses the crude estimate $|\mathcal{G}_N^c| \leq |\mathcal{S}_N| = N!$ as well as the fact that 
	\[
	\sum_{j = 1}^N \| q(\sigma(j)) - \hat{q}(j) \|^2 \geq \frac{N \delta_N^2 \log N}{N^{1/n}}
	\]
	for $\sigma \in \mathcal{G}_N^c$.
	
	\medskip
	
	{\it Step 4.} We now let $\lambda = \lambda_N$ depend on $N$. By the trivial bound $N!\le N^N$, we can bound \eqref{eqn:mixture.weight.estimate} above by
	\begin{equation} \label{eqn:mixture.weight.estimate2}
	C_0 \exp \left[ \alpha' \lambda_N N W_N - \alpha \lambda_N N \frac{\delta_N^2 \log N}{N^{1/n}} + N \log N \right],
	\end{equation}
	where $C_0 > 0$ is a constant. We will choose $\lambda_N$ suitably such that the sum in \eqref{eqn:mixture.weight.estimate} tends to zero exponentially fast as $N \rightarrow \infty$.
	
	Note that $\{q(j),  j \in [N]\}$ and $\{\hat{q}(j),  j \in [N]\}$ are two independent collection of i.i.d.~random vectors sampled from $P_1$ (modulo the relabeling in Step 2 which is irrelevant). Let $V_N = \mathcal{W}_2^2 (L_N(1), P_1)$ and $U_N = \mathcal{W}_2^2 (L_N'(1), P_1)$. By the triangle inequality, we have that
	\begin{equation} \label{eqn:estimate.WN}
	W_N \leq 2(U_N + V_N).
	\end{equation}
	
	{To bound the right hand side of \eqref{eqn:estimate.WN} we apply some known results} on the rate of convergence of $V_N$ (and hence $U_N$). In particular, we will apply Theorem 2 of the paper \cite{FG15} Fournier and Guillin. Since $P_1$ is supported in $\Delta_n$, it has compact support in $\mathbb{R}^n$ and all exponential moments exist. Hence Assumption (1) in their Theorem 2 is satisfied for $p=2$. Consider the function $a(N,x)$, for $0< x < 1$, from that result (replacing their $d$ by $n$): 
	\[
	a(N,x)=C\begin{dcases}
	&\exp\left( -c N x^2\right), \quad \text{if $n < 4$},\\
	&\exp\left( -c N (x/ \log(2+ 1/x))^2\right), \quad \text{if $n=4$},\\
	&\exp\left( - c N x^{n/2}\right), \quad \text{if $n > 4$}.
	\end{dcases}
	\]
	Fix $x>0$ and all large $N$ such that $x/N< 1$. Then 
	\[ 
	a\left(N, (x/N)^{2/n} \right)=C\begin{dcases}
	& \exp\left( - c N^{1-4/n} x^{4/n} \right),\quad \text{if $n < 4$},\\
	& \exp\left( - c x/(\log(2+ (N/x)^{1/2}))^2  \right), \quad \text{if $n=4$},\\
	& \exp\left(- c x \right), \quad \text{if $n > 4$}. 
	\end{dcases}
	\]
	The function $b(n,x)=0$ in \cite[Theorem 2]{FG15} for $x > 1$ under our assumption. 
	
	Then, there exist some positive constants $C, c$ depending on $n, P_1$ such that for all $x>0$ and all large enough $N$ with $x/N< 1$, we have
	\[
	\begin{split}
	P\left( N^{2/n} V_N \ge  x^{2/n} \right) &\le C \exp\left( - c x\right).\\
	%P\left( N^{1/2}V_N \ge  x^{1/2} \right) &\le C \exp\left( - c x\right), \quad n=2.
	\end{split}
	\]
	Replacing $x$ by $\frac{2}{c}\log N$, there exists a constant $c_0$ such that for all large enough $N$, 
	\[
	\begin{split}
	P\left( N^{2/n} V_N \ge  c_0 (\log N)^{2/n} \right) &\le C N^{-2}.
	%P\left( N^{1/2}V_N \ge  c_0 (\log N)^{1/2} \right) &\le C N^{-2}, \quad n=2.
	\end{split}
	\]
	In particular, by the Borel-Cantelli lemma, almost surely, for all large enough $N$, we have
	\eq\label{eq:boundvn}
	V_N \le c_0\left( \frac{\log N}{N}\right)^{2/n}.
	\en 
	Of course, exactly the same statements hold for $U_N$, and hence for $W_N$, by \eqref{eqn:estimate.WN}, perhaps for a different choice of constants. 
	
	\medskip
	
	{\it Step 5.} {Fix $n \geq 2$.} Now choose $\delta^2_N= c_1 N^{-1/n}$, for some large enough constant $c_1>0$. Given any $c_0 >0$, by choosing $c_1$ suitably we can guarantee that, for all $n\ge 2$, 
	\[
	N^{-1/n}\delta_N^2\log N=c_1 N^{-2/n}\log N > c_0 \left( \frac{\log N}{N}\right)^{2/n},
	\]
	for all large enough $N$. In fact, by choosing $c_1$ large enough, almost surely, for all large enough $N$, we can guarantee
	\begin{equation*}
	\begin{split}
	\lambda_N N\left( \frac{\delta_N^2\log N}{N^{1/n}} - \frac{\alpha'}{\alpha}W_N \right) 
	&\ge \lambda_N N\left( c_1 \frac{\log N}{N^{2/n}} - \frac{c_0 \alpha'}{\alpha} \left( \frac{\log N}{N}\right)^{2/n} \right)\\ &\ge \frac{1}{2} \lambda_N N^{1-2/n} \log N.
	\end{split}
	\end{equation*}
	Now let $\lambda_N=\frac{4}{\alpha}N^{2/n}$. For all large enough $N$, we have
	\eq\label{eq:nexpbnd}
	\alpha \lambda_N N \left( \frac{\delta_N^2\log N}{N^{1/n}} - \frac{\alpha'}{\alpha}W_N \right) \ge  2 N \log N.
	\en
	Therefore, from \eqref{eqn:mixture.weight.estimate2}, for all large enough $N$, 
	\eq\label{eq:nexpbnd2}
	\sum_{\sigma \notin \mathcal{G}_N} \nu_N^\sigma \le C_0 \exp\left( - N\log N\right).
	\en

	Combining everything, from \eqref{eqn:bound.Wasserstein}, almost surely, for all large enough $N$, 
	\[
	\mathcal{W}^2_2(M_N, M_N') \leq \frac{c_1+2\log N}{N^{1/n}} + C_0 e^{-N \log N} \leq \frac{3\log N}{N^{1/n}}  , \quad \text{for $n\geq 2$}.
	\]
	
	Since $M_N'$ corresponds to a random sample from the optimal Monge solution, $\mathcal{W}^2_2\left( M_N', R^*\right)$ in the Wasserstein-$2$ metric also satisfies \eqref{eq:boundvn}, perhaps with a different choice of constants. Note that $R^*$ is compactly supported and therefore has all finite exponential moments. Combining this with our last bound, using triangle inequality, and ignoring lower order terms gives us the statement of the theorem. 
\end{proof}

\subsection{Dynamic extension} \label{sec:dynamic}
We now extend the previous static result to a dynamic setting. Let $\mathcal{M}_1(0, 1]$ be the collection of Borel probability measures on $(0, 1]$ equipped with the L\'{e}vy metric of weak convergence. We may regard it as the subset of $\mathcal{M}_1[0, 1]$ of all probability measures that do not charge the singleton $\{0\}$.

Fix $n \geq 2$ and define the subintervals
\[
E_i = \left( \frac{i-1}{n}, \frac{i}{n}\right], \quad i = 1, \ldots, n.
\]
Then there is a natural projection map from $\mathcal{M}_1(0, 1]$ to $\overline{\Delta}_n$ given by
\begin{equation} \label{eqn:measure.projection}
\mu \in \mathcal{M}_1(0, 1] \mapsto (\mu(E_i))_{1 \leq i \leq n} \in \overline{\Delta}_n.
\end{equation}

Let $\mathrm{Leb}$ denote the uniform (Lebesgue) measure on $(0, 1]$. Consider the relative entropy $I(\mu) := H\left( \mathrm{Leb} \mid \mu\right)$ of the Lebesgue measure with respect to $\mu \in \mathcal{M}_1(0, 1]$. Our first observation is the following.

\begin{lemma} \label{lem:entropy.restriction}
	Given $\pi \in \Delta_n$, let $\mathcal{M}_{\pi} \subset \mathcal{M}_1(0, 1]$ denote the collection of $\mu$ such that $\mu(E_i) = \pi_i$ for $1 \leq i \leq n$. Then
	\[
	\inf_{\mu \in \mathcal{M}_{\pi}} H\left( \mathrm{Leb} \mid \mu\right) = H\left( \overline{e} \mid \pi\right),
	\]
	where the right hand side is the discrete relative entropy defined by \eqref{eqn:relative.entropy}.
\end{lemma}
\begin{proof}
	Let $\mu \in \mathcal{M}_{\pi}$. Then the projection of $\mu$ under \eqref{eqn:measure.projection} is $\pi$, and the projection of $\mathrm{Leb}$ is $\overline{e}$. By the information monotonicity of the relative entropy, we have $H\left( \overline{e} \mid \pi\right) \leq H\left( \mathrm{Leb} \mid \mu\right)$. By tensorization, the equality is achieved (uniquely) by $\mu^* \in \mathcal{M}_{\pi}$ such that $\mu^*$ is uniform when restricted to each of the subintervals $E_i$.
\end{proof}

We now define a Lagrangian action on functions on $\Delta_n$ which is consistent with the cost $c$. Let $\mathcal{I}_n$ denote the set of all functions $f = (f_1, \ldots, f_n) : [0, 1] \rightarrow [0, 1]^n$ such that each $f_i$ is right-continuous, strictly increasing, $f_i(0) = 0$, and $f(1) \in \Delta_n$. Each $f_i$ can be thought of as the distribution function of a measure $\mu_i$ supported on the subinterval $E_i$ in the sense that
\[
f_i(t) = \mu_i\left( \frac{i-1}{n}, \frac{i-1 + t}{n}\right], \quad f_i(1) = \mu_i(E_i).
\]
Since $\{E_i\}_{1 \leq i \leq n}$ is a partition of $(0, 1]$, together the coordinate functions represent a probability measure $\mu_f$ in $\mathcal{M}_1(0, 1]$ given by $\mu_f = \sum_{i = 1}^n \mu_i$. Let $\dot{f}_i(s)$ denote the density of the absolutely continuous part of $f_i$ (with respect to $\mathrm{Leb}$) at $s \in (0, 1]$.

\medskip

We now explain how elements of $\mathcal{I}_n$ induce transport paths in $\Delta_n$. Let $p, q \in \Delta_n$ be given, and let $f \in \mathcal{I}_n$ be such that $f(1) = q \odot p^{-1}$. For each $t \in [0, 1]$, let $\pi(t)$ be the unique element in $\simp_n$ such that 
\begin{equation}\label{eqn:pi.t.formula}
\pi_i(t) \propto (1-t)\frac{1}{n} + f_i(t), \quad 1 \leq i \leq n.
\end{equation}
In particular, $\pi(0)=\overline{e}$ and $\pi(1)=f(1)$. Given $f$, we define an interpolation $\{q(t)\}_{0 \leq t \leq 1}$ between $p$ and $q$ by $q(t) = p \odot \pi(t)$, i.e.,
\begin{equation} \label{eqn:q.t.formula}
q_i(t)= \frac{(1-t) p_i/n + f_i(t) p_i}{ (1-t)/n + \sum_{j=1}^n f_j(t) p_j}, \quad 1 \leq i \leq n.
\end{equation}
The cost of transporting $p$ to $q$ along this path is taken to be $H\left(\leb\mid \mu_f \right)$. We formalize the above discussion with the following definition.

\begin{defn} [Dynamic cost function] \label{def:Lagrangian}
	Let $p, q \in \Delta_n$. Consider a path $\{q(t)\}_{0 \leq t \leq 1}$ of the form \eqref{eqn:q.t.formula} for some $f \in \mathcal{I}_n$ with $f(1) = q \odot p^{-1}$. We define the transport cost of this path by the Lagrangian action
	\begin{equation} \label{eqn:Lagrangian}
	I(f) := I(\mu_f) = H\left( \mathrm{Leb} \mid \mu_f \right) = -\frac{1}{n} \sum_{i = 1}^n \int_0^1 \log \dot{f}_i(s) ds,
	\end{equation}
	{when $\mathrm{Leb} \ll \mu_f$ and $I(f) = \infty$ otherwise.}
\end{defn}

From Lemma \ref{lem:entropy.restriction}, we immediately have the following

\begin{prop} \label{prop:Lagrangian}
	For given $p, q \in \Delta_n$, let $\pi = p^{-1} \odot q$. Then
	\begin{equation} \label{eqn:dynamic.cost.inf}
	\inf\{ I(f) : f(1) = \pi\} = H\left( \overline{e} \mid \pi\right) = c(p, q),
	\end{equation}
	and the infimum is attained uniquely by $f(t) = t\pi$, $t \in [0, 1]$.
\end{prop}

By Proposition \ref{prop:Lagrangian}, the unique cost minimizing interpolating path, given the two end points, satisfies $f(t)=t\pi(1)$ and
\eq\label{eq:interpol}
\pi(t)= (1-t) \overline{e} + t (q \odot p^{-1}), \quad q(t)= p \odot \pi(t),\quad q(0)=p,\; q(1)=q.
\en
Hence, $\pi(\cdot)$ is the linear interpolation between the barycenter of $\simp_n$ and $\pi(1)$. This is the same as the displacement interpolation (see Remark \ref{rmk:straight.line}).

Our Lagrangian action \eqref{eqn:Lagrangian} should be compared with the classical integrated kinetic energy
\begin{equation} \label{eqn:KE}
\int_0^1 \|\dot{\omega}(t)\|^2 dt, \quad \omega(0) = x, \ \omega(1) = y,
\end{equation}
corresponding to the quadratic cost $\|x - y\|^2$ (here $\dot{\omega}$ is the velocity). The minimizing curves in \eqref{eqn:KE} are constant-velocity straight lines. In \eqref{eqn:Lagrangian}, it is the portfolio weights that travel along constant-velocity straight lines.

Now we formulate the dynamic extension to the multiplicative particle system described in Section \ref{sec:particle.system}. Recall that the standard gamma subordinator is a right-continuous, increasing L\'evy process $\{\gamma(t)\}_{t \geq 0}$ such that $\gamma(0)=0$ and $\gamma(t)$ is distributed as $\mathrm{Gamma}(t, 1)$. The gamma subordinator can be thought of as a (random) measure on $(0,\infty)$. It has no mass at zero, since, almost surely, $\gamma(0+)=0$ by right continuity. 

For $\lambda >0$, we will normalize this measure to get the family of Dirichlet processes as a random probability measure on the interval $(0,1]$. More formally, given $\lambda>0$, define the {\diri} process $D^{\lambda}$ as a (random) element in $M_1(0,1]$ such that 
\eq\label{eq:gammabridge}
D^{\lambda}\left(a,b   \right]= \frac{\gamma(b \lambda) - \gamma(a \lambda)}{\gamma(\lambda)},\quad 0< a < b\le 1. 
\en
That is, $D^{\lambda}$ is the measure with the distribution function $\gamma(t \lambda)/\gamma(\lambda)$, $0 \leq t \leq 1$. Interestingly, this process is also the conditional process given $\gamma(\lambda)=1$ (see \cite{EY04}). We will often write $D^{\lambda}(t)$ to denote the distribution function $D^{\lambda}[0,t]=D^{\lambda}(0,t]$. The large deviations of the Dirichlet process, as $\lambda \rightarrow \infty$, connects it with our transport problem. The following result is originally due to Lynch and Sethuraman \cite{LS87} and the following statement is taken from \cite{F07}. See in particular Theorem 4.7 (for $\alpha=0$) and Lemma 4.5 (for $\nu=\leb$). 

\begin{lemma}\label{lem:dirldp} The family of laws $\left( D^{\lambda},\; \lambda \ge 0 \right)$ satisfies a Large Deviation Principle (LDP) on ${\mathcal{M}}_1(0,1]$ with speed $\lambda$ and a good rate function given by $I(\mu) = H\left( \leb \mid \mu\right)$. 
\end{lemma}

We can now extend Theorem \ref{thm:static.transport} to this dynamic setting. Consider independent i.i.d.~samples $\{p(j),\; j\ge 1\}$ from $P_0$ and $\{q(j),\; j \ge 1\}$ from $P_1$. %Also consider the joint distributions $M_N^\sigma$ with the corresponding mixture probabilities $\nu_N^\sigma$ for $\sigma \in \perm_N$.
For each $\sigma \in \perm_N$ we define 
\[
\pi^\sigma(j) = q(\sigma(j)) \odot \left(p(j)\right)^{-1}, \quad j \geq 1.
\]
Extend the underlying probability space $\left( \Omega, \mathcal{F}, \mathbb{P} \right)$ to include an i.i.d.~sequence $\left\{ \gamma(j),\; j \geq 1 \right\}$ where each $\gamma(j)$ is an $n$-dimensional vector of independent standard gamma subordinators $\left(\gamma_1(j,t), \ldots, \gamma_n(j,t),\; t\ge 0\right)$. By an abuse of notation we retain the same notation $\left( \Omega, \mathcal{F}, \mathbb{P} \right)$ for the extended probability space. For each $N \geq 1$, define processes $\left(\pi^\sigma(i,t),\; 0\le t \le 1,\;1 \leq i \leq N \right)$, $0\le t \le 1$, by 
\eq\label{eq:gammasubrep}
\pi_i^\sigma(j,t) = \frac{\gamma_i(j,t \lambda)}{\gamma_i(j, \lambda)} \pi^\sigma_i(j), \quad 1 \leq i \leq n, \; 1 \leq j \leq N.
\en
Then each $\pi^\sigma(j,\cdot)\in \inc_n$, and therefore, as described above, can be thought of as the distribution function of a random element in $M_1(0,1]$. For notational brevity, let us denote that random measure also by $\pi^\sigma(j, \cdot)$. The context will make it clear whether we refer to the measure or its distribution function.  

\begin{lemma}\label{lem:conddir}
	The law of $\pi^\sigma(j,\cdot)$ on $\mathcal{M}_1(0,1]$ is the conditional distribution of $D^{n \lambda}$, given the events
	\[
	D^{n \lambda}\left( E_i\right)= \pi^\sigma_i(j), \quad 1 \leq i \leq n.
	\]
\end{lemma}
\begin{proof}
	We drop the index $j$ for this proof. Suppose $\gamma$ is standard gamma subordinator. Then, by stationary independent increment property 
	\[
	\gamma_i(t):= \gamma\left( (i-1)n + t \right) - \gamma((i-1)n), \quad 0\le t \le \lambda, \; 1 \leq i \leq n,
	\]
	are $n$ independent gamma subordinators run on the time interval $[0, \lambda]$. The event $D^{n \lambda}(E_i)=\pi^\sigma_i$ is equivalent to the event $\gamma_i(\lambda)= \pi^\sigma_i$. By the conditional independence property of the Dirichlet process mentioned below \eqref{eq:gammabridge} the vector of measures with distribution functions $\left( \gamma_i(\cdot \lambda)/\gamma_i(\lambda),\; 1 \leq i \leq n\right)$ are jointly independent of each other and also independent of the random vector $\left(\gamma_i(\lambda),\; 1 \leq i \leq n \right)$. In particular, conditioning on the latter has no effect on the former. This completes the proof. 
\end{proof}

For $\sigma \in \mathcal{S}_N$ and $j \in [N]$, let 
\begin{equation} \label{eqn:conditional.path}
q^\sigma(j,t):= p(j) \odot \pi^\sigma(j,t),
\end{equation}
where $\pi^\sigma(j,\cdot)$ is sampled from the conditional distribution given in Lemma \ref{lem:conddir} for the given initial and end points. Let $M^{\sigma}_N$ be the empirical path
\[
M^\sigma_N:=  \frac{1}{N} \sum_{j=1}^N \delta_{\{q^\sigma(j, \cdot) \}}.
\]
It is a probability measure on the space of RCLL paths from $[0,1]$ to $\simp_n$. That is, with probability $\frac{1}{N}$ it chooses the pair $(p(j), q(\sigma(j)))$ and the path $q^\sigma(j, \cdot)$ given by \eqref{eqn:conditional.path}.

Analogous to \eqref{eqn:M.N}, we let
\[
M_N := \sum_{\sigma \in \mathcal{S}_N} \nu_N^{\sigma} M^{\sigma}_N,
\]
where the weights $\nu_N^{\sigma}$ are given as in \eqref{eqn:mixture.weight}. The interpretation is the same, except that the pair in the discrete system is replaced by a (random) path.

\begin{thm}\label{thm:dynconv}
	Under the assumptions of Theorem \ref{thm:static.transport}, $\mathbb{P}$ almost surely, $M_N^{\sigma}$ converges weakly to the delta mass on the path $(q(t),\; 0\le t\le 1)$ given by \eqref{eq:interpol}, where $(p,q)$ is chosen at random from the optimal Monge coupling of $P_0$ and $P_1$. 
\end{thm}

\begin{proof} Consider the representation \eqref{eq:gammasubrep}. By the well-known strong law of large numbers applied to the gamma subordinator, ${\pi^\sigma_j(i,t)}/{\pi^\sigma_j(i)}$, as a monotone function in $t\in (0,1]$, converges uniformly almost surely to the function $h(t)=t$, independent of $\sigma$, as $\lambda \rightarrow \infty$. Since, by Theorem \ref{thm:static.transport} and the continuous mapping theorem, the law of $\pi^\sigma$ converges to the portfolio map of the optimal Monge coupling of $P_0$ and $P_1$, the statement of the theorem follows.
\end{proof}

\section{Entropy along displacement interpolation} \label{sec:entropy}
Consider the displacement interpolation $\{P_t\}_{0 \leq t \leq 1}$, given in Definition \ref{def:displacement.interpolation}, for a pair of probability measures $P_0, P_1 \in \mathcal{L}_a$. In this section we study the behaviors of the entropy along such paths.

\subsection{Statement of main result}
By Theorem \ref{thm:transport.solution} there exists an exponentially concave function $\varphi_1$ such that its portfolio map $\boldsymbol{\pi}_1$ induces the optimal Monge coupling
\[
q = T_1(p) = p \odot \boldsymbol{\pi}_1(p^{-1}).
\]
To focus on the main ideas we will impose some regularity conditions on the function $\varphi$; similar conditions are also adopted in \cite{PW16, W17}. We believe they can be relaxed using the ideas of \cite[Section 4]{M97} but this will not be attempted in this paper.

\begin{asmp} \label{ass:interpolation.assumptions} { \ }
	Assume that $\varphi_1$ is twice continuously differentiable everywhere on the primal simplex $\Delta_n$, and that the quadratic form $L(r)$ defined by \eqref{eqn:matrix.L} is strictly positive definite everywhere on $\Delta_n$.
\end{asmp}

Recall the reference measure $\mu_0$ from Definition \ref{def:entropic.measure} and the Euclidean coordinate system $(p_1, \ldots, p_{n-1})$ with range $\mathcal{D}_{n-1}$.

\begin{defn}[Entropy] 
	Let $P \in \mathcal{P}(\Delta_n)$ be absolutely continuous with respect to $\mu_0$. Let $\rho = \frac{dP}{d\mu_0}$ be the density of $P$ with respect to $\mu_0$. Then, the entropy of $P$ with respect to $\mu_0$ is defined by
	\begin{equation} \label{eqn:entropy}
	\mathrm{Ent}_{\mu_0}(P) =  \int_{\mathcal{D}_{n-1}} \rho(p) \log \rho(p) d\mu_0(p).
	\end{equation}
\end{defn}

Following the convention in the theory of optimal transport, our entropy is the negative of the information-theoretic entropy.

\medskip
We are now ready to state the main result of this section.

\begin{thm} \label{thm:entropy}
	Consider the displacement interpolation $\{P_t\}_{0 \leq t \leq 1}$, defined by \eqref{eqn:displacement.interpolation}, under Assumption \ref{ass:interpolation.assumptions}. Then the map
	\begin{equation} \label{eqn:entropy.convex}
	t \mapsto \mathrm{Ent}_{\mu_0} (P_t) + n \mathbf{C}(P_0, P_t)
	\end{equation}
	is convex on $[0, 1]$.
\end{thm}

\begin{rmk} [Discussion] \label{rmk:convexity.discussion}
	While the entropy itself may not be convex along the displacement interpolation, Theorem \ref{thm:entropy} states that it becomes convex after adding $n$ times the transport cost from $P_0$ to $P_t$. It is well known that Ricci curvature lower bounds of a Riemannian manifold can be characterized by the semiconvexity of entropy with respect to the $2$-Wasserstein displacement interpolation (see \cite{V08} for an in-depth treatment). Since our cost function is an asymmetric divergence rather than a squared distance, this theory does not apply to our transport cost ${\bf C}(P, Q)$. On the other hand, the geometry induced by a divergence on a manifold has been studied extensively in information geometry \cite{A16}. Recently, we showed in \cite{PW16} that the cost function \eqref{eqn:cost.function} induces on $\Delta_n$ an information geometry with constant sectional curvature $-1$, and we suspect that this property is related to the convexity result \eqref{eqn:entropy.convex}. We believe that Theorem \ref{thm:entropy} is a first step towards a theory of generalized geometric structures on space of probability distributions.
\end{rmk}

\subsection{A Monge-Amp\`{e}re equation} \label{sec:Monge.Ampere}
The main ingredient of the proof of Theorem \ref{thm:entropy} is a Monge-Amp\`{e}re equation which relates the measures $P_0$ and $P_t = (T_t)_{\#} P_0$ in our transport problem. Namely, if we write
\begin{equation} \label{eqn:interpolation.densities}
dP_0(p) = \rho_0(p) d\mu_0(p) \quad \text{and} \quad dP_t(q) = \rho_t(q) d\mu_0(q),
\end{equation}
where $\rho_0$ and $\rho_t$ are respectively the densities of $P_0$ and $P_t$ with respect to the reference measure $\mu_0$ (see \eqref{eqn:entropic.measure}), we want to express $\rho_t$ in terms of $\rho_0$ and the transport map. 

We begin by introducing some notations. Fix $0 \leq t \leq 1$. For $p \in \Delta_n$ we let $r = p^{-1}$. Recall that $\varphi_t = (1 - t) \varphi_0 + t \varphi_1$, where $\varphi_0(r) = \frac{1}{n} \sum_{i = 1}^n \log r_i$ and $\varphi_1$ is given by Theorem \ref{thm:transport.solution}. Let
\[
\tilde{r} = (r_1, \ldots, r_{n-1}) \in \mathcal{D}_{n-1}
\]
be the (first $n - 1$) coordinates of $r$ and write $\varphi_t(r) = \tilde{\varphi}_t(\tilde{r})$ as a function of $\tilde{r}$. Also let
\[
\tilde{L}_t(\tilde{r}) := - \nabla^2 \tilde{\varphi}_t( \tilde{r}) - (\nabla \tilde{\varphi}_t(\tilde{r}))(\nabla \tilde{\varphi}_t(\tilde{r}))^{\top}
\]
be the Riemannian matrix of the $L$-divergence of $\tilde{\varphi}_t$ under coordinate system $\tilde{r}$. Abusing notations, we also write $\tilde{L}_t(\tilde{r}) = \tilde{L}_t(r)$. By Assumption \ref{ass:interpolation.assumptions} we have $\mathrm{det}(\tilde{L}_t(\tilde{r})) > 0$ for all $\tilde{r} \in \mathcal{D}_{n-1}$. Now we may state the equation which is quite different from the classical Monge-Amp\`{e}re equation for the quadratic cost (see for example \cite[Theorem 4.8]{V03}).

\begin{thm} \label{thm:Monge.Ampere}
	Let $p \in \Delta_n$, $r = p^{-1}$ and $q =  T_t(p)$. Write $\pi = (\pi_1, \ldots, \pi_n) = \boldsymbol{\pi}_t(r)$. Using the notations of \eqref{eqn:interpolation.densities}, we have
	\begin{equation} \label{eqn:Monge.Ampere}
	\rho_t(q) =  \rho_0(p) \frac{\pi_1  \cdots \pi_n}{\mathrm{det}(\tilde{L}_t(r))} \frac{1}{r_1^2 \cdots r_n^2}.
	\end{equation}
\end{thm}

\begin{rmk}
	The bold font $\boldsymbol{\pi}_t$ is reserved for the portfolio map at time $t$, while the usual $\pi_i$ refers to the $i$th coordinate of the vector $\pi$ (see Notation \ref{not:portfolio.maps}). Strictly speaking, we should write $\pi_i$ as $(\boldsymbol{\pi}_t(r))_i$. This makes \eqref{eqn:Monge.Ampere} overly cumbersome, which is why we have chosen this notational abuse. 
\end{rmk}

\begin{rmk}
	The regularity of the Dirichlet transport, which is closely related to the equation \eqref{eqn:Monge.Ampere}, was studied in the recent papers \cite{KZ19, WY19}. Specifically, it was shown that the Ma-Trudinger-Wang tensor associated with the cost function \eqref{eqn:cost.function} vanishes identically. As a consequence, the transport maps $\{T_t\}_{0 \leq t \leq 1}$ of the displacement interpolation are smooth under suitable conditions on $P_0$ and $P_1$. 
\end{rmk}

The proof of Theorem \ref{thm:Monge.Ampere} will make use of the following lemmas.

\begin{lemma} \label{lem:change.variables}
	Let $p \in \Delta_n$, $r = p^{-1}$ and $q = T_t(p)$. Let $u \in (0, \infty)^{n-1}$ be the vector defined by
	\begin{equation} \label{eqn:r.to.u}
	u =  {\bf 1} + \frac{\nabla \tilde{\varphi}_t(\tilde{r})}{1 - \tilde{r}^{\top} \nabla \tilde{\varphi}_t(\tilde{r})},
	\end{equation}
	where ${\bf 1} = (1, \ldots, 1)^{\top}$ is the vector of all ones. Then
	\begin{equation} \label{eqn:u.to.q}
	q_i = \frac{u_i}{1 + \sum_{j = 1}^{n-1} u_j}, \quad i = 1, \ldots, n - 1.
	\end{equation}
\end{lemma}
\begin{proof}
	Note that
	\[
	\frac{\partial}{\partial \tilde{r}_i} = \frac{\partial}{\partial r_i} - \frac{\partial}{\partial r_n}, \quad i = 1, \ldots, n - 1.
	\]
	Since
	\[
	e_i - r = (e_i - e_n) - \sum_{k = 1}^{n-1} r_k (e_k - e_n),
	\]
	we have
	\[
	\nabla_{e_i - r} \varphi_t(r) =
	\begin{dcases}
	\frac{\partial \tilde{\varphi}}{\partial \tilde{r}_i} (\tilde{r}) - \sum_{k = 1}^{n-1} \tilde{r}_k \frac{\partial \tilde{\varphi}}{\partial \tilde{r}_k}(\tilde{r}), & \text{for } i = 1, \ldots, n - 1,\\
	- \sum_{k = 1}^{n-1} \tilde{r}_k \frac{\partial \tilde{\varphi}}{\partial \tilde{r}_k}(\tilde{r}), & \text{for } i = n.
	\end{dcases}
	\]
	Consider the vector of weight ratios given by
	\[
	{\bf w}_t(r) = \left(\frac{(\boldsymbol{\pi}_t(r))_1}{r_1}, \ldots, \frac{(\boldsymbol{\pi}_t(r))_n}{r_n}\right)^{\top}.
	\]
	Since $\frac{(\boldsymbol{\pi}_t(r))_i}{r_i} = 1 + \nabla_{e_i - r} \varphi_t(r)$ by \eqref{eqn:portfolio.map}, we have
	\begin{equation} \label{eqn:weight.ratio.as.gradient}
	\left( ({\bf w}_t(r))_1, \ldots, ({\bf w}_t(r))_{n-1}\right)^{\top} = {\bf 1} + \nabla \tilde{\varphi}(\tilde{r}) - (\tilde{r}^{\top} \nabla \tilde{\varphi}(\tilde{r})) {\bf 1},
	\end{equation}
	and $({\bf w}_t(r))_n = 1 - \tilde{r}^{\top} \nabla \tilde{\varphi}(\tilde{r})$.
	
	We have $q = T_t(p) = p \odot \boldsymbol{\pi}_t(r)$. By \eqref{eqn:q.as.weight.ratio}, we have
	\[
	\frac{q_i}{q_n} = \frac{({\bf w}_t(r))_i}{({\bf w}_t(r))_n}, \quad i = 1, \ldots, n - 1.
	\]
	Let $u = (q_1/q_n, \ldots, q_{n-1}/q_n)^{\top}$. From \eqref{eqn:weight.ratio.as.gradient}, we have
	\[
	u = \frac{{\bf 1} + \nabla \tilde{\varphi}(\tilde{r}) - (\tilde{r}^{\top} \nabla \tilde{\varphi}(\tilde{r})) {\bf 1}}{1 - \tilde{r}^{\top} \nabla \tilde{\varphi}(\tilde{r})} = {\bf 1} + \frac{\nabla \tilde{\varphi}(\tilde{r})}{1 - \tilde{r}^{\top} \nabla \tilde{\varphi}(\tilde{r})}.
	\]
	Since $u_i = q_i/q_n$, we obtain \eqref{eqn:u.to.q} by a straightforward computation.
\end{proof}

We also recall the so-called matrix determinant lemma. A proof can be found in \cite[Lemma 1.1]{DZ07}.

\begin{lemma} [Matrix determinant lemma]
	Let $A$ be an invertible  $m \times m$ square matrix and $u, v \in \mathbb{R}^m$ be column vectors. Then
	\[
	\mathrm{det}(A + uv^{\top}) = (1 + v^{\top} A^{-1} u) \mathrm{det}(A).
	\]
\end{lemma}

\begin{lemma} \label{lem:det.Jt}
	Let $J_t(r)$ be the Jacobian matrix of the transformation $\tilde{r} \mapsto u$ \eqref{eqn:r.to.u}. Then
	\begin{equation} \label{eqn:Jt.det}
	|\mathrm{det}(J_t(r))| = \frac{r^n_n}{(\boldsymbol{\pi}_t(r))_n^n} \mathrm{det} (\tilde{L}_t(r)).
	\end{equation}
	It follows that the map $\tilde{r} \mapsto u$ (and hence the transport map $T_t$) is a $C^1$-diffeomorphism.
\end{lemma}
\begin{proof}
	Let $\partial_i = \frac{\partial}{\partial \tilde{r}_i}$ and similarly for the second derivatives. Writing down the components of \eqref{eqn:r.to.u} explicitly, we have
	\[
	u_i = 1 + \frac{\partial_i \tilde{\varphi}_t}{1 - \sum_{k = 1}^{n-1} \tilde{r}_k \partial_k \tilde{\varphi}_t}.
	\]
	For notational simplicity we write $\pi = \boldsymbol{\pi}_t(r)$, $w = \pi / r$ and suppress the argument $r$ throughout. Differentiating, we have
	\begin{equation*}
	\begin{split}
	\frac{\partial u_i}{\partial \tilde{r}_j} &= \frac{(1 - \sum_{k = 1}^{n-1} \tilde{r}_k \partial_k \tilde{\varphi}_t ) \partial_{ij} \tilde{\varphi} + \partial_i \tilde{\varphi} ( \sum_{k = 1}^{n-1} \tilde{r}_k \partial_{kj} \tilde{\varphi}_t + \partial_j \tilde{\varphi} )}{(1 - \sum_{k = 1}^{n-1} \tilde{r}_k \partial_k \tilde{\varphi}_t)^2}\\
	&= \frac{1}{w_n} \partial_{ij} \tilde{\varphi}_t + \frac{1}{w_n^2} \partial_i \tilde{\varphi} \left( \sum_{k = 1}^{n-1} \tilde{r}_k \partial_{kj} \tilde{\varphi}_t + \partial_j \tilde{\varphi}_t\right),
	\end{split}
	\end{equation*}
	where $w_n = 1 - \sum_{k = 1}^{n-1} \tilde{r}_k \partial_k \tilde{\varphi}_t$ as in the line after \eqref{eqn:weight.ratio.as.gradient}.
	
	Expressing the above in matrix form, we have
	\[
	J_t(r) = \frac{1}{w_n} \nabla^2 \tilde{\varphi}_t + \frac{1}{w_n^2} \nabla \tilde{\varphi}_t ( \nabla^2 \tilde{\varphi}_t \tilde{r} + \nabla \tilde{\varphi}_t)^{\top}.
	\]
	Now we apply the matrix determinant lemma (with $A = \frac{1}{w_n} \nabla^2 \tilde{\varphi}_t$, $u = \frac{1}{w_n} \nabla \tilde{\varphi}_t$ and $v = \frac{1}{w_n} ( \nabla^2 \tilde{\varphi}_t \tilde{r} + \nabla \tilde{\varphi}_t)$) to get
	\begin{equation} \label{eqn:matrix.det.apply1}
	\begin{split}
	\mathrm{det}(J_t(r)) &= \frac{1}{w_n^{n-1}} \left(1 + \frac{1}{w_n} (\tilde{r}^{\top} \nabla^2 \tilde{\varphi}_t + \nabla \tilde{\varphi}_t^{\top}) (\nabla^2 \tilde{\varphi}_t)^{-1} (\nabla \tilde{\varphi}_t) \right) \mathrm{det} (\nabla^2 \tilde{\varphi}) \\
	&= \frac{1}{w_n^n} \left(1 + (\nabla \tilde{\varphi}_t)^{\top} (\nabla^2 \tilde{\varphi}_t)^{-1} (\nabla \tilde{\varphi}_t)\right) \mathrm{det} (\nabla^2 \tilde{\varphi}).
	\end{split}
	\end{equation}
	In the last equality we used the identity $w_n = 1 - \tilde{r}^{\top} \nabla \tilde{\varphi}_t$.
	
	On the other hand, again by the matrix determinant lemma, we have
	\begin{equation*}
	\begin{split}
	\mathrm{det}(\tilde{L}_t(r)) &= \mathrm{det} \left( - \nabla^2 \tilde{\varphi}_t - (\nabla \tilde{\varphi}_t)(\nabla \tilde{\varphi}_t)^{\top} \right) \\
	&= \left(1 + (\nabla \tilde{\varphi}_t)^{\top} (\nabla^2 \tilde{\varphi}_t)^{-1} (\nabla \tilde{\varphi}_t) \right)  \mathrm{det} (-\nabla^2 \tilde{\varphi}).
	\end{split}
	\end{equation*}
	Plugging this into \eqref{eqn:matrix.det.apply1} gives the formula \eqref{eqn:Jt.det}.
	
	From the proof of \cite[Proposition 2.9]{W17}, we have that the map $\tilde{r} \mapsto u$ is $C^1$ and one-to-one. By Assumption \ref{ass:interpolation.assumptions}, the Jacobian determinant is everywhere non-zero. Thus, by the inverse function theorem, the map $\tilde{r} \mapsto u$ is a $C^1$-diffeomorphism.
\end{proof}

\begin{proof}[Proof of Theorem \ref{thm:Monge.Ampere}]
	Consider the transformation $p \mapsto q = T_t(p)$ which is a $C^1$-diffeomorhpism by Lemma \ref{lem:det.Jt}. By the change of variables formula, we have
	\begin{equation} \label{eqn:Monge.Ampere.computing}
	\rho_t(q) = \rho_0(p) \frac{q_1 \cdots q_n}{p_1 \cdots p_n} \frac{1}{\left| \frac{\partial (q_1, \ldots, q_{n-1})}{\partial (p_1, \ldots, p_{n-1})} \right|}, \quad q = T_t(p).
	\end{equation}
	
	It remains to find the Jacobian determinant of the transformation $p \mapsto q$. Using the notations of Lemma \ref{lem:change.variables}, the transport map can be written as the composition
	\begin{equation} \label{eqn:compositions}
	p \mapsto r = p^{-1} \mapsto u \mapsto q.
	\end{equation}
	Thus we can express the Jacobian determinant as a product.
	
	First we consider $p \mapsto r$. Since
	\[
	r_i = \frac{1/p_i}{\sum_{j = 1}^n 1/p_j},
	\]
	for $1 \leq i, j \leq n - 1$ we have
	\[
	\frac{\partial r_i}{\partial p_j} = \frac{-r_i}{p_i} \delta_{ij} + \frac{1}{p_i} (r_j^2 - r_n^2).
	\]
	By the matrix determinant lemma, we have after some computations
	\begin{equation} \label{eqn:determinant1}
	\left| \frac{\partial (r_1, \ldots, r_{n-1})}{\partial (p_1, \ldots, p_{n-1})}\right| = \frac{r_1 \cdots r_n}{p_1 \cdots p_n}.
	\end{equation}
	
	The Jacobian determinant of the transformation $r \mapsto u$ has been computed in Lemma \ref{lem:det.Jt}.
	
	Finally, it is easy to show that
	\begin{equation} \label{eqn:determinant2}
	\frac{\partial (q_1, \ldots, q_{n-1})}{\partial (u_1, \ldots, u_{n-1})} = q_n^n.
	\end{equation}
	
	Combining \eqref{eqn:determinant1}, Lemma \ref{lem:det.Jt} and \eqref{eqn:determinant2}, we have
	\[
	\left| \frac{\partial (q_1, \ldots, q_{n-1})}{\partial (p_1, \ldots, p_{n-1})} \right| = \frac{r_1 \cdots r_n}{p_1 \cdots p_n} \frac{r^n}{\pi_n^n} \mathrm{det} (\tilde{L}_t(r)) q_n^n.
	\]
	Plugging this into \eqref{eqn:Monge.Ampere.computing}, we get
	\begin{equation*}
	\begin{split}
	\rho_t(q) &= \rho_0(p) \frac{q_1 \cdots q_n}{r_1 \cdots r_n} \frac{\pi_n^n / r_n^n}{\mathrm{det} (\tilde{L}_t(r)) q_n^n}.
	\end{split}
	\end{equation*}
	Since $q_i = (\pi_i/r_i) / \sum_{j = 1}^n (\pi_j/r_j)$ by \eqref{eqn:q.as.weight.ratio}, simplifying gives the desired formula \eqref{eqn:Monge.Ampere}.
\end{proof}

\subsection{Proof of Theorem \ref{thm:entropy}}
Consider the entropy
\[
\mathrm{Ent}_{\mu_0}(P_t) = \int \log \frac{dP_t}{d\mu_0}(q) d P_t(q).
\]
Using the Monge-Amp\`{e}re equation \eqref{eqn:Monge.Ampere}, we have
\begin{equation*}
\begin{split}
\mathrm{Ent}_{\mu_0}(P_t) &= \int \log \rho_t(T_t(p)) dP_0(p) \\
&= \int \log \left( \rho_0(p) \frac{\pi_1 \cdots \pi_n}{\mathrm{det} (\tilde{L}_t(r))} \frac{1}{r_1^2 \cdots r_n^2} \right) dP_0,
\end{split}
\end{equation*}
where $\pi = \boldsymbol{\pi}_t(r) = \boldsymbol{\pi}_t(p^{-1})$. It follows that $\mathrm{Ent}_{\mu_0}(P_t)$ equals
\[
\mathrm{Ent}_{\mu_0} (P_0) + \int \sum_{i = 1}^n \log \pi_i dP_0 - \int \log \mathrm{det} (\tilde{L}_t(r)) d P_0
\]
plus a constant which does not depend on $t$.

On the other hand, since by Lemma \ref{lem:cost.as.entropy}
\[
c(p, q) = H\left( \overline{e} \mid \pi \right) = \sum_{i = 1}^n \frac{1}{n} \log \frac{1/n}{\pi_i},
\]
we have
\[
n \mathbf{C}(P_0, P_t) = n \log \frac{1}{n} - \int \sum_{i = 1}^n \log \pi_i d P_0.
\]
Thus
\[
\mathrm{Ent}_{\mu_0}(P_t) + n \mathbf{C}(P_0, P_t) = K -  \int \log \mathrm{det}(\tilde{L}_t(r)) dP_0(p).
\]
for some constant $K$, and the convexity of $t \mapsto \mathrm{Ent}_{\mu_0}(P_t) + n \mathbf{C}(P_0, P_t)$ is equivalent to that of
\[
t \mapsto -  \int \log \mathrm{det}(\tilde{L}_t(r)) dP_0(p).
\]

Recall the L\"owner order on the cone of positive semidefinite matrices where $A \preceq B$ if $B-A$ is positive semidefinite and $A\prec B$ if $B-A$ is positive definite.

\begin{lemma} \label{eqn:Lt.concave}
	For any $r \in \Delta_n$ fixed, the map $t \mapsto \tilde{L}_t(r)$ is concave in the L\"{o}wner order, i.e., if $t = (1 - \alpha) t_1 + \alpha t_2$ and $\alpha \in [0, 1]$, then
	\[
	\tilde{L}_t(r) - \left[ (1 - \alpha) \tilde{L}_{t_1}(r) + \alpha \tilde{L}_{t_2}(r)\right]
	\]
	is positive semidefinite.
\end{lemma}
\begin{proof}
	Since $\tilde{\varphi}_t = (1 - \alpha) \tilde{\varphi}_{t_1} + \alpha \tilde{\varphi}_{t_2}$, we have
	\begin{equation*}
	\begin{split}
	& \tilde{L}_t(r) - \left[ (1 - \alpha) \tilde{L}_{t_1}(r) + \alpha \tilde{L}_{t_2}(r) \right] \\
	&= \alpha(1 - \alpha) \left[ (\nabla \tilde{\varphi}_{t_1})(\nabla \tilde{\varphi}_{t_1})^{\top} + (\nabla \tilde{\varphi}_{t_1}) (\nabla \tilde{\varphi}_{t_2})^{\top} +\right.  \\
	&\quad \quad \quad \quad \quad \quad \left. (\nabla \tilde{\varphi}_{t_2}) (\nabla \tilde{\varphi}_{t_1})^{\top} +  (\nabla \tilde{\varphi}_{t_2}) (\nabla \tilde{\varphi}_{t_2})^{\top} \right] \\
	&= \alpha (1 - \alpha) (  \nabla \tilde{\varphi}_{t_1} + \nabla \tilde{\varphi}_{t_2}) (  \nabla \tilde{\varphi}_{t_1} + \nabla \tilde{\varphi}_{t_2})^{\top},
	\end{split}
	\end{equation*}
	which is clearly positive semidefinite.
\end{proof}

By the previous lemma the map $t \mapsto \tilde{L}_t(z)$ is concave in the L\"owner order. The map $A \mapsto (\mathrm{det}(A))^{1/(n-1)}$ is non-decreasing in the L\"owner order on the space of positive semidefinite matrices by the Minkowski determinant inequality (see \cite[Theorem 7.8.8, page 482]{horn1990}). Also, it is a well-known fact (see for example \cite[Theorem 17.9.1]{CT06}) that $-\log \mathrm{det} (\cdot)$ is a convex function of positive semidefinite matrices. Combining, $-\log \mathrm{det} (\cdot)$ is a non-increasing convex function in the L\"owner order. By Lemma \ref{eqn:Lt.concave}, $t \mapsto - \log \mathrm{det}(\tilde{L}_t(r))$ is convex in $t$ and the theorem is proved.

\section{Dimension-free bounds of the transport cost} \label{sec:bounds}
As a consequence of ths structure of our Lagrangian, in this final section we derive, under suitable conditions, upper bounds of the transport cost $\mathbf{C}(P, Q)$ that do not depend explicitly on the dimension $n$ (or, rather, $n - 1$, of the simplex $\Delta_n$). 
%{\color {blue} Although we have not yet found specific applications of this bound, we believe it is of independent and suggests that the infinite dimensional version \eqref{eqn:dynamic.cost.inf} should be further investigated.}
There are few models of sequentially generating random elements from the unit simplices of increasing dimension that satisfy some natural consistency condition. One natural model is to take i.i.d. positive random variables $X_1, X_2, \ldots, X_n$ and divide each coordinate by the total sum $S=X_1+ \ldots + X_n$ to get a random vector $(X_1/S, \ldots, X_n/S)$ in $\Delta_n$. For example, the uniform distribution on $\Delta_n$ occurs this way when we take each $X_i$ to be exponential with rate one. However, we do not know how to analyze this model sequence.%}

%First, we can generate a random vector $(\theta_1, \ldots, \theta_n)$ with values in $\mathbb{R}^n$ and then apply the map ${\bf s}: \mathbb{R}^n \rightarrow \Delta_n$ given by
%\[
%{\bf s}(\theta) = \left(\frac{e^{\theta_1}}{\sum_{j = 1}^n e^{\theta_j}}, \ldots, \frac{e^{\theta_n}}{\sum_{j = 1}^n e^{\theta_j}} \right).
%\]
%For example, if $e^{\theta_1}, \ldots, e^{\theta_n}$ are i.i.d.~exponential with rate $1$, then ${\bf s}(\theta)$ is uniformly distributed on $\Delta_n$. Another common family is the logit-normal where $\theta$ has a multivariate normal distribution.

Instead we take the following related model of generating a random element from the simplex $\Delta_n$. 
Consider i.i.d. continuous random variables on the unit interval $(0, 1)$, say $U_1, \ldots, U_{n-1}$. Arrange them in increasing order:
\[
U_{(0)}=0 < U_{(1)} < U_{(2)} < \cdots < U_{(n-1)} < U_{(n)}=1.
\]
Consider the gaps between the order statistics, i.e.,
\begin{equation} \label{eqn:gaps}
p_i = U_{(i)} - U_{(i-1)}, \quad i = 1, \ldots, n,
\end{equation}
where by convention $u_{(0)} := 0$ and $u_{(n)} := 1$. Then $p = (p_1, \ldots, p_n) \in \Delta_n$. For instance, we can again recover the uniform distribution on $\Delta_n$ if we let $U_1, \ldots, U_{n-1}$ to be i.i.d.~uniform random variables on $(0, 1)$. Moreover, this construction is structurally aligned with the infinite dimensional version \eqref{eqn:dynamic.cost.inf}, an extension that should be further investigated.

%\subsection{Gaps between order statistics}
Given $n \geq 2$, let $P_n \in \mathcal{P}(\Delta_n)$ denote the uniform distribution on $\Delta_n$. On the other hand, let $X_1, \ldots, X_{n-1}$ be i.i.d.~$[0, 1]$-valued random variables with distribution function $F$ that admits a continuous strictly positive density $f$. In particular, $F(0)=0$ and $F(1)=1$. Let $q = (q_1, \ldots, q_n)$ be the gaps between the order statistics of $X$, and let $Q_n$ denote its law in $\Delta_n$. A natural coupling of $P_n$ and $Q_n$ can be obtained by generating $U_1, \ldots, U_{n-1}$ whose gaps are distributed according to $P_n$ and defining $X_i=F^{-1}(U_i)$ for $1 \leq i \leq n - 1$. More explicitly, let $Z = (Z_1, \ldots, Z_n) \sim P_n$ be a random vector that is uniformly distributed on $\Delta_n$. Then we may write
\[
U_{(i)} = Z_1 + \cdots + Z_i, \quad 1 \leq i \leq n, \quad U_{(0)} = 0.
\]
The coupling is then given by
\begin{equation} \label{eqn:PnQn.coupling}
\begin{split}
q_i &= F^{-1}(U_{(i)}) - F^{-1}(U_{(i-1)}) \\
  &= F^{-1}(Z_1 + \cdots + Z_i) - F^{-1}(Z_1 + \cdots + Z_{i - 1}).
\end{split}
\end{equation}

\begin{remark}
By varying the distribution function $F$, the possible distributions $Q_n$ that the transform \eqref{eqn:PnQn.coupling} generates form a somewhat restricted subset of $\mathcal{P}(\Delta_n)$.\footnote{We thank an anonymous referee for pointing out this point.} For example, the coupling \eqref{eqn:PnQn.coupling} cannot generate a distribution $Q_n$ such that $q_1 \leq q_2  \leq \cdots \leq q_n$ almost surely. To see this, suppose $n = 3$. If $U_{(1)}$ is close to $0$ and $U_{(2)}$ is close to $1$ (which is possible as $Z$ is uniform), then by the continuity and strict monotonicity of $F^{-1}$, we have that $q_1$, $q_3$ are close to $0$ but $q_2$ is close to $1$. 

\end{remark}

Using this coupling (which is generally sub-optimal for the cost function $c$) we obtain the following dimension-free bound.

\begin{thm} \label{thm:second.bound}
	Suppose that the density $f$ of $F$ is continuous and strictly positive on $[0,1]$. Then
	\[
	\limsup_{n \rightarrow \infty} \cost\left(P_n, Q_n \right) \le H( F),
	\]
	where the right side is the Shannon entropy of the distribution function $F$ with respect to the Lebesgue measure:
	\[
	H(F) = -\int_0^1 f(u) \log f(u) du.
	\] 
\end{thm}

\begin{proof}
Since $F^{-1}$ is an increasing function, for the coupling given before the statement the order statistics are preserved. That is, $F^{-1}\left( U_{(i)}\right)=X_{(i)}$. Thus the cost of transport is given by
\begin{equation*}
\begin{split}
	c(p, q) &= \log\left( \frac{1}{n} \sum_{i=1}^n \frac{F^{-1}(U_{(i)}) F^{-1}(U_{i-1})}{U_{(i)} - U_{(i-1)}}\right)\\
	&\quad - \frac{1}{n}\sum_{i=1}^n \log\left(  \frac{F^{-1}(U_{(i)}) - F^{-1}(U_{i-1})}{U_{(i)} - U_{(i-1)}}\right),
\end{split}
\end{equation*}
where $U_{(0)}\equiv 0$ and $U_{(n)}\equiv 1$. The proof is completed by the following lemma. 
\end{proof}

\begin{lemma}\label{lem:convl1} The sequence of random variables 
	\[
	\log\left( \frac{1}{n} \sum_{i=1}^n \frac{F^{-1}(U_{(i)}) - F^{-1}(U_{i-1})}{U_{(i)} - U_{(i-1)}}\right), \quad n\ge 1,  
	\]
	converges to zero in $L^1$. The sequence of random variables
	\[
	-\frac{1}{n}\sum_{i=1}^n \log\left(  \frac{F^{-1}(U_{(i)}) - F^{-1}(U_{i-1})}{U_{(i)} - U_{(i-1)}}\right)
	\] 
	converges in $L^1$ to $\Ent(F)=-\int_0^1 f(u) \log f(u) du$. 
\end{lemma}
\begin{proof}
	Let $G=F^{-1}$. Since $f$ is strictly positive and continuous on $[0,1]$, it is uniformly continuous on $[0,1]$ and bounded above by, say, $M>0$, and bounded below by, say, $m >0$. Hence $G$ is continuously differentiable and $G'(u)=1/f\left( G(u)\right) \in \left[1/M, 1/m\right]$. 
	
	Consider the function
	\[
	R(h) := \max_{x \in [0, 1 - h]} \left| \frac{G(x + h) - G(x)}{h} - G'(x) \right|, \quad 0 < h < 1.
	\]
	We have the straightforward estimate
	\[
	R(h) \leq \max_{x \in [0, 1 - h]} \frac{1}{h} \int_0^h \left| G'(x + t) - G'(x) \right| dt .
	\]
	Since $G'$ is uniformly continuous on $[0, 1]$, for $\epsilon > 0$ there exists $\delta > 0$ such that $|G'(x + t) - G'(x)| < \epsilon$ whenever $t \leq \delta$. It follows that
	\[
	\lim_{h \downarrow 0} R(h) = 0.
	\]
	
	Let $\epsilon > 0$ be given. Let $\delta > 0$ be such that $R(h) < \epsilon$ whenever $h < \delta$. Then
	\begin{equation*}
	\begin{split}
	& \left| \frac{1}{n} \sum_{i = 1}^n \frac{G(U_{(i)}) - G(U_{(i - 1)})}{U_{(i)} - U_{(i-1)}} - \frac{1}{n} \sum_{i = 1}^n G'(U_{(i - 1)})\right| \\
	&\leq \epsilon + \frac{1}{n} \sum_{i : U_{(i)} - U_{(i - 1)} > \delta} G'(U_{(i - 1)}) \\
	&\leq \epsilon + \left(\frac{1}{m}\right) \cdot \frac{\#\{i : U_{(i)} - U_{(i - 1)} > \delta\}}{n}.
	\end{split}
	\end{equation*}
	Since the $U_{(i)}$'s are the order statistics of the uniform distribution, it is not difficult to show that $\frac{\#\{i : U_{(i)} - U_{(i - 1)} > \delta\}}{n} \rightarrow 0$ in $L^1$. Thus $\frac{1}{n} \sum_{i = 1}^n \frac{G(U_{(i)}) - G(U_{(i - 1)})}{U_{(i)} - U_{(i-1)}}$ and $\frac{1}{n} \sum_{i = 1}^n G'(U_{(i - 1)})$ have the same $L^1$ limit (if exists).
	
	On the other hand, since $\frac{1}{M} \leq G' \leq \frac{1}{m}$, we have
	\[
	\frac{1}{n} \sum_i G'(U_{(i)}) \rightarrow \int_0^1 G'(u) du = F^{-1}(1) - F^{-1}(0) = 1
	\]
	almost surely. By continuity of $\log$, we also have
	\[
	\log\left( \frac{1}{n} \sum_{i=1}^n \frac{G(U_{(i)}) - G(U_{(i-1)})}{ U_{(i)} - U_{(i-1)}} \right) \rightarrow 0 \quad a.s.
	\]
	Since $G'$ is bounded between $ \frac{1}{M}$ and $\frac{1}{m}$, by the mean value theorem the convergence holds in $L^1$ as well.
	
	Since $G'$ is bounded above and is bounded below from $0$, a similar argument shows that $-\frac{1}{n}  \sum_{i = 1}^n \log \left( \frac{G(U_{(i)}) - G(U_{(i - 1)})}{U_{(i)} - U_{(i-1)}}\right)$ and $-\frac{1}{n} \sum_{i = 1}^n \log G'(U_{(i - 1)})$ have the same $L^1$-limit given by
	\[
	- \int_0^1 \log G'(u) du = \int_0^1 \log f(G(u)) du = -\int_0^1 f(x) \log f(x) dx,
	\]
	which is the entropy of $f$.

\end{proof}

\section*{Appendix}\label{sec:appendix} 

\begin{proof}[Proof of Lemma \ref{prop:class.L}]
	Let $\theta_i = - \log p_i$ and $\phi_i = - \log q_i$ for $1 \leq i \leq n$. Then the cost function \eqref{eqn:cost.function} takes the form
	\begin{equation} \label{eqn:cost.exponential}
	c(p, q) = \log \left( \frac{1}{n}\sum_{i = 1}^n  e^{\theta_i - \phi_i}\right) - \frac{1}{n}  \sum_{i = 1}^n (\theta_i - \phi_i).
	\end{equation}
	%Here our notation is chosen to be consistent with the exponential coordinate system used in \cite{PW16}.
	By the Cauchy-Schwarz inequality, we have 
	\begin{equation*}
	\begin{split}
	c(p, q) &\leq \frac{1}{2} \left[ \log \left( \frac{1}{n}\sum_{i = 1}^n  e^{2\theta_i}\right) - \frac{1}{n} \sum_{i = 1}^n (2 \theta_i) \right] \\ 
	&\quad + \frac{1}{2} \left[ \log \left( \frac{1}{n}\sum_{i = 1}^n e^{-2\phi_i}\right) - \frac{1}{n} \sum_{i = 1}^n (-2 \phi_i) \right].
	\end{split}
	\end{equation*}
	Since
	\[
	\log \left( \frac{1}{n} \sum_{i = 1}^n e^{2\theta_i}\right) \leq 2 \max_{1 \leq i \leq n} |\theta_i| \leq 2 \sum_{i = 1}^n |\theta_i|,
	\]
	we have the estimate
	\[
	c(p, q) \leq \left(1 + \frac{1}{n}\right) \left( \sum_{i = 1}^n |\theta_i| + |\phi_i|\right).
	\]
	Integrating against any coupling $R \in \Pi(P, Q)$ and replacing the constant (which is irrelevant) by $2$ shows that the transport cost is finite whenever $P, Q \in \mathcal{L}$. 
\end{proof}

\begin{proof}[Proof of Theorem \ref{thm:transport.solution}]
	%\todo{This explanation is not enough. We are using different coordinate systems in two papers. We should explain how to convert one to the other.}
	Since $P, Q \in \mathcal{L}$, by Proposition \ref{prop:class.L} we have $\mathbf{C}(P, Q) < \infty$. Since the cost function is continuous and bounded below, by general results of optimal transport (see for example \cite{V03, V08}), there exists an optimal coupling $R^* \in \Pi(P, Q)$ solving the transport problem, and its support is $c$-cyclical monotone.
	
	Let $m \geq 1$ and let $\{(p(s), q(s)\}_{s = 0}^{m - 1}$ be a sequence in the support of $R^*$. By the $c$-cyclical monotonicity of $R^*$, we have
	\[
	\sum_{s = 0}^{m-1} \log \left( \frac{1}{n} \sum_{i = 1}^n \frac{q_i(s)}{p_i(s)} \right) \leq \sum_{s = 0}^{m-1} \log \left( \frac{1}{n} \sum_{i = 1}^n \frac{q_i(s)}{p_i(s + 1)} \right),
	\]
	where by convention $(p(m), q(m)) := (p(0), q(0))$. For each $s$ let $\pi(s) = q(s) \odot p(s)^{-1}$ and $r(s) = p(s)^{-1}$. Rearranging, we have
	\begin{equation*} 
	\begin{split}
	& \sum_{s = 0}^{m-1} \log \left( \sum_{i = 1}^n \frac{q_i(s)/p_i(s)}{\sum_{k = 1}^n q_k(s)/p_k(s)} \frac{p_i(s)}{p_i(s + 1)} \right) \\
	&= \sum_{s = 0}^{m-1} \log \left( \sum_{i = 1}^n \pi_i(s) \frac{r_i(s + 1)}{r_i(s)} \right) \geq 0.
	\end{split}
	\end{equation*}
	Thus the (multi-valued) portfolio map
	\[
	r \mapsto \{\pi = q \odot p^{-1} : p = r^{-1}, (p, q) \in \mathrm{supp}(R^*)\}
	\]
	induced by the optimal coupling is multiplicatively cyclical monotone in the sense of \eqref{eqn:MCM}. (In \cite[Proposition 12]{PW14} we performed this argument using another coordinate system.)
	
	By \cite[Proposition 4, Proposition 6]{PW14}, there exists an exponentially concave function $\varphi$ on $\Delta_n$ such that if $\boldsymbol{\pi}$ is the portfolio map generated by $\varphi$, $(p, q)$ is any pair in the support of $R^*$ and $\varphi$ is differentiable at $r = p^{-1}$, then
	\begin{equation} \label{eqn:transport.as.portfolio}
	\pi = q \odot p^{-1} = \boldsymbol{\pi}(r).
	\end{equation}
	Rearranging, we have $q = p \odot \boldsymbol{\pi}(p^{-1})$ which is the image of $p$ under the mapping \eqref{eqn:deterministic.transport}. Since $P \in \mathcal{L}_a$ is absolutely continuous and $\varphi$ is differentiable almost everywhere, for $P$-a.e.~values of $p$ there is a unique element $q \in \Delta_n$ such that $(p, q) \in \mathrm{supp}(R^*)$ and \eqref{eqn:transport.as.portfolio} holds. This proves both (i) and (ii).
\end{proof}

\begin{proof}[Proof of Proposition \ref{prop:interpolation}]
	First we show that $P_t \in \mathcal{L}$ for all $t$. By Remark \ref{rmk:straight.line}, for each $p$ the trace of $\{T_t(p)\}_{0 \leq t \leq 1}$ is a straight line in $\Delta_n$. It follows that for each $i$ we have
	\begin{equation*}
	\begin{split}
	|\log (T_t(p))_i | &\leq \max\{ |\log p_i|, | \log (T_1(p))_i|\} \leq |\log p_i| + | \log (T_1(p))_i|.
	\end{split}
	\end{equation*}
	Since both $P_0, P_1 \in \mathcal{L}$ by assumption, we have $P_t \in \mathcal{L}$ as well.
	
	Next we prove that $P_t$ is absolutely continuous. For vectors $a$ and $b$ we let $\frac{a}{b} = (\frac{a_i}{b_i})$ be the vector of component-wise ratios, and we use $a \cdot b$ and $\langle a, b \rangle$ interchangeably to denote the Euclidean dot product. 
	
	Let $0 < t < 1$ be given.  Let ${\bf w}_t(r) = \frac{\boldsymbol{\pi}_t(r)}{r}$ be the vector of unnormalized weight ratios. Recall that $q = T_t(p) = p \odot \boldsymbol{\pi}_t(p^{-1}) = r^{-1} \odot \boldsymbol{\pi}_t(r)$ and similarly for $q'$. Then, by \eqref{eqn:q.as.weight.ratio}, we have
	\[
	q = T_t(p) = \left( \frac{({\bf w}_t(r))_i}{\sum_{j = 1}^n ({\bf w}_t(r))_j}\right)_{1 \leq i \leq n}.
	\]
	Thus, if we can prove that the distribution $\tilde{P}_t$ of ${\bf w}_t(r)$ (where $r = p^{-1}$ and $p \sim P_0$) is absolutely continuous, then $P_t$ is absolutely continuous and we are done.
	
	To this end, consider the quantity
	\begin{equation} \label{eqn:inner.product}
	\begin{split}
	& \left\langle \frac{\boldsymbol{\pi}_t(r')}{r'} - \frac{\boldsymbol{\pi}_t(r)}{r}, r' - r \right\rangle \\
	&= (1 - t) \left\langle \frac{\overline{e}}{r'} - \frac{\overline{e}}{r}, r' - r \right\rangle + t \left\langle \frac{\boldsymbol{\pi}_1(r')}{r'} - \frac{\boldsymbol{\pi}_1(r)}{r}, r' - r \right\rangle \\
	&= (1 - t) \left( 2 - \overline{e} \cdot \frac{r}{r'} - \overline{e} \cdot \frac{r'}{r}  \right) +  t \left(2 - \boldsymbol{\pi}_1(r') \cdot \frac{r}{r'} - \boldsymbol{\pi}_1(r) \cdot \frac{r'}{r} \right) \\
	&\leq - (1 - t) \log \left( \left( \overline{e} \cdot \frac{r}{r'} \right)\left( \overline{e} \cdot \frac{r'}{r} \right) \right) - t \log \left(\left( \boldsymbol{\pi}_1(r') \cdot \frac{r}{r'} \right) \left( \boldsymbol{\pi}_1(r) \cdot \frac{r'}{r} \right)\right).
	\end{split}
	\end{equation}
	In the last line we used the estimate $\log(1 + x) \leq x$.
	
	By the multiplicative cyclical monotonicity of the portfolio maps (see \eqref{eqn:MCM}), we have
	\[
	\left( \overline{e} \cdot \frac{r}{r'} \right)\left( \overline{e} \cdot \frac{r'}{r} \right)  \geq 1, \quad \left( \boldsymbol{\pi}_1(r') \cdot \frac{r}{r'} \right) \left( \boldsymbol{\pi}_1(r) \cdot \frac{r'}{r} \right) \geq 1
	\]
	for all $r, r' \in \Delta_n$. It follows from \eqref{eqn:inner.product} and the Cauchy-Schwarz inequality that
	\begin{equation} \label{eqn:inverse.Lipschitz}
	\left\|  \frac{\boldsymbol{\pi}_t(r')}{r'} - \frac{\boldsymbol{\pi}_t(r)}{r} \right\| \geq (1 - t) \frac{\log \left( \left( \overline{e} \cdot \frac{r}{r'} \right)\left( \overline{e} \cdot \frac{r'}{r} \right) \right)}{\|r' - r\|}, \quad r \neq r'.
	\end{equation}
	By \eqref{eqn:L.divergence.cost}, the right hand side of \eqref{eqn:inverse.Lipschitz} equals
	\begin{equation} \label{eqn:weight.lower.bound}
	(1 - t) \frac{c(r, r') + c(r', r)}{\|r - r'\|},
	\end{equation}
	which is positive for $r \neq r'$. By the Taylor approximation \eqref{eqn:L.divergence.approx} $c(r, r') + c(r', r)$ is of order $\|r - r'\|^2$ when $r \approx r'$, thus \eqref{eqn:weight.lower.bound} is of order $(1 - t) \|r - r'\|$ when $r \approx r'$.
	
	From \eqref{eqn:inverse.Lipschitz} and the previous observation, the mapping $p \mapsto r = p^{-1} \mapsto {\bf w}_t(r)$ is one-to-one and its inverse is locally Lipschitz. Since $P_0$ is absolutely continuous by assumption, we have that $\tilde{P}_t$, and hence $P_t$, is absolutely continuous.
	
	To prove the second claim, let $\boldsymbol{\pi}_t$ be the portfolio map at time $t$. By Lemma \ref{lem:cost.as.entropy}, we have
	\begin{equation*}
	\begin{split}
	\mathbf{C}(P_0, P_t) &= \mathbb{E}_{p \sim P_0} \left[ H\left( \overline{e} \mid \boldsymbol{\pi}_t(p^{-1})\right) \right] \\
	&= \mathbb{E}_{p \sim P_0} \left[ H\left( \overline{e} \mid (1 - t) \overline{e} + t \boldsymbol{\pi}_1(p^{-1})\right) \right].
	\end{split}
	\end{equation*}
	By properties of the relative entropy (see for example \cite[Theorem 2.7.2]{CT06}) the quantity $H\left( \overline{e} \mid (1 - t) \overline{e} + t \boldsymbol{\pi}_1(p^{-1})\right)$ is smooth and convex in $t$, and is increasing and strictly convex whenever $\boldsymbol{\pi}_1(p^{-1}) \neq \overline{e}$. Since $P_0 \neq P_1$ by assumption, the last condition holds on a set of positive probability under $P_0$. This completes the proof of the proposition.
\end{proof}

\begin{proof}[Proof of Lemma \ref{lem:suff.conditions}]
	Recall that 
	\begin{equation} \label{eqn:L.divergence.rewrite}
	{\bf D}[q' : q] = \log (1 + \nabla \varphi(q) \cdot (q' - q) ) - (\varphi(q') - \varphi(q)).
	\end{equation}
	Since $\log (1 + x) \leq x$, we have the upper bound
	\[
	{\bf D}[q' : q] \leq \nabla \varphi(q) \cdot (q' - q) - (\varphi(q') - \varphi(q))
	\]
	which is the Bregman divergence of $\varphi$ (see \cite[Chapter 1]{A16}). Let $q \neq q'$. Applying Taylor's theorem along the line segment $[q, q']$ from $q$ to $q'$, we have
	\[
	{\bf D}[q' : q] \leq \|q' - q \|^2 (v^{\top} (-\nabla^2 \varphi(q'')) v)
	\]
	for some $q''$ on $[q, q']$ and $v = \frac{q' - q}{\| q' - q \|}$. From the hypotheses we have
	\[
	(v^{\top} (-\nabla^2 \varphi(q'')) v) \leq C_1,
	\]
	so the upper bound in \eqref{eqn:L.divergence.bound} holds with $\alpha' = C_1$.
	
	To derive a lower bound, let $\Phi = e^{\varphi}$ and express \eqref{eqn:L.divergence.rewrite} in the form
	\begin{equation*}
	\begin{split}
	{\bf D}[q' : q] &= \log \left( \frac{\Phi(q) + \nabla \Phi(q) \cdot (q' - q)}{\Phi(q')}\right) \\
	&= - \log \left( \frac{\Phi(q) + \nabla \Phi(q) \cdot (q' - q) + \|q' - q \|^2 (v^{\top} \nabla^2 \Phi(q'') v)}{\Phi(q) + \nabla \Phi(q) \cdot (q' - q) } \right) \\
	&= - \log \left( 1 + \frac{(v^{\top} \nabla^2 \Phi(q'') v)}{\Phi(q) + \nabla \Phi(q) \cdot (q' - q)} \|q' - q\|^2\right).
	\end{split}
	\end{equation*}
	Again $q''$ is some point on $[q, q']$ and $v$ is as above. Using $- \log(1 + x) \geq -x$, we have the bound
	\begin{equation} \label{eqn:lower.bound.estimate}
	{\bf D}[q' : q] \geq \frac{C_2}{\Phi(q) + \nabla \Phi(q) \cdot (q' - q)} \|q' - q\|^2.
	\end{equation}
	Since $\Phi$ is non-negative and concave on $\Delta_n$, it is bounded above by some $M > 0$.  Since $\|q' - q\| \leq 1$ for $q, q' \in \Delta_n$, we have
	\[
	\Phi(q) +  \nabla \Phi(q) \cdot (q' - q) \leq M + C_3.
	\]
	Plugging this into \eqref{eqn:lower.bound.estimate} gives the lower bound with $\alpha = \frac{C_2}{M + C_3}$.
\end{proof}

\begin{acknowledgements}
S.~P.~thanks Martin Huesmann for very useful discussions.
\end{acknowledgements}


\bibliographystyle{spmpsci}      % mathematics and physical sciences
\bibliography{infogeo.new}   % name your BibTeX data base

%\begin{thebibliography}{}
%%
%% and use \bibitem to create references. Consult the Instructions
%% for authors for reference list style.
%\end{thebibliography}

\end{document}